\documentclass[12pt,reqno]{amsart}
\usepackage[headings]{fullpage}
\usepackage{amssymb,amsmath,amscd,tikz,tikz-cd,mathtools}
\usepackage{graphicx}
\usepackage{texdraw}
\usepackage{pb-diagram}
\usepackage[all,cmtip]{xy}
\usepackage{url}
\usepackage[bookmarks=true,%
    colorlinks=true,%
    linkcolor=blue,%
    citecolor=blue,%
    filecolor=blue,%
    menucolor=blue,%
    urlcolor=blue,%
    breaklinks=true]{hyperref}
\usepackage{slashed}
\usepackage{stmaryrd}
\usepackage{verbatim}

\usetikzlibrary{hobby}
\usetikzlibrary{decorations.pathmorphing}
\tikzset{snake it/.style={decorate, decoration=snake}}
\usetikzlibrary{patterns}

\newtheorem{theorem}{Theorem}[section]
\theoremstyle{definition}
\newtheorem{proposition}[theorem]{Proposition}
\newtheorem{lemma}[theorem]{Lemma}
\newtheorem{definition}[theorem]{Definition}
\newtheorem{remark}[theorem]{Remark}
\newtheorem{corollary}[theorem]{Corollary}

\newtheorem{question}[theorem]{Question}

\def\BZ{\mathbb Z}
\def\BQ{\mathbb Q}
\def\BR{\mathbb R}
\def\BC{\mathbb C}

\def\calA{\mathcal A}

\def\calC{\mathcal C}

\def\calH{\mathcal H}

\def\s{\sigma}

\def\SL{\mathrm{SL}}

\def\pt{\partial}

\def\Res{\mathrm{Res}}
\def\Vol{\mathrm{Vol}}
\def\a{\alpha}
\def\b{\beta}

\def\th{\theta}

\def\Li{\mathrm{Li}}

\def\be{\begin{equation}}
\def\ee{\end{equation}}

\def\z{\zeta}

\def\diag{\mathrm{diag}}

\def\Om{\Omega}

\usepackage{accents}

\def\={\;=\;}

\newcommand{\om}{\omega}

\def\dR{\mathrm{dR}}
\def\vphi{\varphi}
\def\Sf{f^{\mathrm{sym}}}
\def\Sh{\mathrm{S} h}

\def\nv{\mathrm{naiv}}
\def\W{\mathsf{W}}
\def\rem{\mathrm{rem}}
\def\tr{\mathrm{tr}}
\def\Spec{\mathrm{Spec}}
\def\max{\mathrm{max}}

\makeatletter

\renewcommand\thepart{\@Roman\c@part}%
\renewcommand\part{%
   \if@noskipsec \leavevmode \fi
   \par
   \addvspace{6.7ex}%
   \@afterindentfalse
   \secdef\@part\@spart}
\def\@part[#1]#2{%
    \ifnum \c@secnumdepth >\m@ne
      \refstepcounter{part}%
      \addcontentsline{toc}{part}{Part~\thepart.\ #1}%
    \else
      \addcontentsline{toc}{part}{#1}%
    \fi
    {\parindent \z@ \raggedright
     \interlinepenalty \@M
     \normalfont
     \ifnum \c@secnumdepth >\m@ne
       \centering\large\scshape \partname~\thepart.%
       \hspace{1ex}%
     \fi%
     \large\scshape #2%
     \markboth{}{}\par}%
    \nobreak
    \vskip 4.7ex
    \@afterheading}
  \def\@spart#1{
  \refstepcounter{part}%
  \addcontentsline{toc}{part}{#1}%
    {\parindent \z@ \raggedright
     \interlinepenalty \@M
     \normalfont
     \centering\large\scshape #1\par}%
     \nobreak
     \vskip 4.7ex
     \@afterheading}
\renewcommand*\l@part[2]{%
  \ifnum \c@tocdepth >-2\relax
    \addpenalty\@secpenalty
    \addvspace{0.75em \@plus\p@}%
    \begingroup
      \parindent \z@ \rightskip \@pnumwidth
      \parfillskip -\@pnumwidth
      {\leavevmode
       \normalsize \bfseries #1\hfil \hb@xt@\@pnumwidth{\hss #2}}\par
       \nobreak
       \if@compatibility
         \global\@nobreaktrue
         \everypar{\global\@nobreakfalse\everypar{}}%
      \fi
    \endgroup
  \fi}

\def\l@subsection{\@tocline{2}{0pt}{2pc}{6pc}{}}
\makeatother

\begin{document}
\title[Explicit classes in Habiro cohomology]{
 Explicit classes in Habiro cohomology}

\author{Stavros Garoufalidis}
\address{
  International Center for Mathematics, Department of Mathematics \\
  Southern University of Science and Technology \\
  Shenzhen, China \newline
  {\tt \url{http://people.mpim-bonn.mpg.de/stavros}}}
\email{stavros@mpim-bonn.mpg.de}

\author{Campbell Wheeler}
\address{Institut des Hautes Études Scientifiques\\ \newline
         35 rte de Chartres, 91440 Bures-sur-Yvette, France \newline
         {\tt \url{https://www.ihes.fr/~wheeler/}}}
\email{wheeler@ihes.fr}


\date{21 May 2025}

\begin{abstract}
  We propose a cycle description of the Habiro cohomology of a smooth variety $X$
  over the spectrum $B$ of an \'etale $\BZ[\lambda]$-algebra and construct explicit
  nontrivial cycles using either the Picard-Fuchs equation on $X/B$
  of a hypergeometric motive, or a push-forward of elements
  of the Habiro ring of $X/B$. In particular, we give explicit classes for
  1-parameter Calabi--Yau families.
  The $q$-hypergeometric origin of our cycles imply
  that they generate $q$-holonomic modules that define $q$-deformations of the
  classical Picard-Fuchs equation. We illustrate our theorems with three examples:
  the Legendre family of elliptic
  curves, the $A$-polynomial curve of the figure eight knot, and
  for the quintic three-fold, whose $q$-Picard Fuchs equation appeared in
  its genus $0$-quantum $K$-theory. Our methods give a unified treatment of 
  quantum $K$-theory and complex Chern-Simons theory around higher dimensional
  critical loci. 
\end{abstract}

\maketitle

{\footnotesize
\tableofcontents
}


\section{Introduction}
\label{sec.intro}

\subsection{A dream}
\label{sub.dream}

A long-held dream of Peter Scholze foresaw the association of a module to a
smooth variety $X$ over the spectrum $B$ of an \'etale $\BZ[\lambda]$-algebra
$R$,\footnote{here $\lambda$ denotes a finite and possibly empty set of variables}
now known as the Habiro cohomology $\calH(X/B)$. This module is finitely generated
(or even finite free after some mild modifications) over the Habiro ring
$\calH_R$ of $R$. A crucial step in this direction was the definition of the
Habiro ring of $R$~\cite{GSWZ}.

Recently, Ferdinand Wagner constructed such a cohomology theory with the usual
functorial properties~\cite{Wagner:habiro}. The Habiro cohomology is a more
sophisticated version of the $q$-de Rham cohomology~\cite{Scholze:canonical,BS:prisms}
and in some sense, a deformation of the de Rham and Hodge cohomology around all
complex roots of unity.

Although it is well-known how to give algebraic de Rham cohomology classes,
say of an affine variety, it seems a much more difficult problem to describe
classes of the $q$-de Rham, $q$-Hodge, or of the Habiro cohomology itself.
Some first steps in this direction were taken in~\cite{Shirai}.
This problem was already encountered in the case of the Habiro ring of an
\'etale $\BZ$-algebra $R$, and elements were given in~\cite{GSWZ} using explicit
proper $q$-hypergeometric series, known as Nahm sums, that appear naturally 
in complex Chern--Simons theory in 3-dimensions and in the admissible series of
Kontsevich--Soibelman.

\subsection{An experiment}
\label{sub.exp}

We first explain how we came up with the definition of the naive Habiro cohomology
given below. The definition was revealed by numerical experiments of asymptotics
of certain $q$-series. 
Drawing on the ideas of~\cite{GSWZ}, we used as input the asymptotic series that
appear in complex Chern--Simons theory and more precisely in the asymptotics of the
3D-index~\cite{GW:periods}. For the simplest hyperbolic $4_1$, we found that its
vertical asymptotics when $q=e^{2\pi i\tau}$ and $\tau$ tends to 0 is in the
imaginary axis sum of 3 asymptotic series
$\tau^{-1/2}\Phi(2\pi i \tau) -i \tau^{-1/2}\Phi(-2\pi i \tau) + \Psi(2 \pi i \tau)$
where 
\be
\label{Phival}
\Phi(\hbar) \=
e^{\frac{2 \Vol(4_1)}{\hbar}}
\frac{1}{3^{\frac{3}{4}}2^{\frac{1}{2}}}
\left(1-\frac{19}{24\sqrt{-3}^3} \hbar
    +\frac{1333}{1152\sqrt{-3}^6} \hbar^2
    -\frac{1601717}{414720\sqrt{-3}^9} \hbar^3+\dots\right) 
\ee
is (ignoring the exponential and the constant term) a series with coefficients
in $\BQ(\sqrt{-3})$, and
\be
\label{Psival}
\Psi(\hbar) \= \frac{\kappa}{\hbar}\Big(1+\frac{227}{155250}\hbar^4
-\frac{3191201}{2200668750}\hbar^8+\cdots\Big) + \hbar \kappa'
\Big(1-\frac{75199}{310500}\hbar^2+\frac{319906757}{1257525000}\hbar^4+\cdots \Big)
\ee
where $\kappa$ and $\kappa'$ are some constants defined shortly. 

The series $\Phi(\hbar)$ is a familiar series of the type studied in~\cite{GSWZ}.
The coefficients in the algebraic number field $\BQ(\sqrt{-3})$
(the trace field of the hyperbolic knot) 
have more than factorially growing universal denominators of
~\cite[Eqn.(9.2)]{GZ:kashaev}. Its symmetrization $\Phi(\hbar)\Phi(-\hbar)$
after a change of variables from $\hbar=\log q$ to $q-1$, is a power series in
$\BZ[\sqrt{-3},\tfrac{1}{6}][\![q-1]\!]$ and in fact an element of the Habiro ring
$\calH_{\BZ[\sqrt{-3},\tfrac{1}{6}]}$. 

On the other hand, the series $\Psi(\hbar)$ has two new features. The first is that
the constants $\kappa$ and $\kappa'$ are periods of an elliptic curve $X$
defined by the equation
\be
\label{A41}
xy+(1-x^2y)(1-y)=0
\ee (the so-called $A$-polynomial of the $4_1$ knot), and given
explicitly by
\be
\kappa \= \int_{\mathcal{C}}\omega \=1.400603042\dots, \qquad
\kappa' \=  \int_{\mathcal{C}}\omega' \= 0.008005298\dots \,,
\ee
where $\calC$ is a suitable contour, 
\be
\label{om2}
\omega\=\frac{1}{\delta}\frac{dx}{x}\,,\qquad
\omega'\=\frac{1}{\delta^{7}}
(x^3 - x^2 - 2x + 5 - 2x^{-1} - x^{-2} + x^{-3})\frac{dx}{x}\,,
\ee
and $\delta^2(x) \=x^{2} - 2x - 1 - 2x^{-1} + x^{-2}$. 

The second feature is that the coefficients of $\Psi(\hbar)$ are rational numbers
with denominators growing like those of $1/k!$, with the coefficient of $\hbar^{150}$
having denominator $2^{-1}\cdot 3^{151}\cdot5^{15}\cdot7\cdot11\cdot23\cdot31\cdot151!$
and contrary to the expectations of~\cite{GSWZ},
after the change of variables from $\hbar=\log q$ to $q-1$, it still has
denominators, albeit slowly growing. For example, the denominator of $(q-1)^{100}$
in $\Psi(\hbar)$ is
\be
2^6\cdot 3^{147}\cdot 5^{125}\cdot\prod_{p=7\,\text{prime}}^{101}p\,.
\ee
The primes $2$, $3$ and $5$ are explained by the bad reduction of the elliptic curve,
but the rest of the primes in the denominator appear with exponent $1$.
The exponents do grow though,
for instance going further to $(q-1)^{150}$, we found that the denominator has
$11$-valuation $2$, and $p$-valuation $1$ for all primes $p$ with $5<p \leq 151$. 

At this point, we removed the largely irrelevant contour $\calC$, and replaced the
constants $\kappa$ and $\kappa'$ by their integrands, namely the differential forms
$\om$ and $\om'$, considered up to exact forms in algebraic de Rham cohomology 
in $H^1_\dR(X/\Spec(\BZ[1/30]),\BQ)$, which we abbreviate simply by $H^1_\dR(X)$.
Doing so, we obtained a power series in $H^1_\dR(X)[\![q-1]\!]$ with mildly growing
denominators, which is therefore $p$-adically convergent for $|q-1|<1$ and can be
expanded around $q-\z_p$ for any prime $p>5$. We then expanded the 3D-index around
$\z_p$ to obtain a second series in  $H^1_\dR(X)[\z_p][\![q-\z_p]\!]$. Comparing the
two, after Frobenius twist for $p=7$, and we found that they matched! This motivated
the definition of the naive Habiro cohomology. 

The upshot of the experiment is that the asymptotics of the 3D-index of the
$4_1$ knot are expressed in terms of elements of the naive Habiro cohomology
$\calH^1_\nv(X)$ where $X$ is the $A$-polynomial of the knot. However, our methods
can define not only one element of $\calH^1_\nv(X)$, but a whole sequence of
them index by integers (the so-called descendants). And then we found that any
three such elements satisfied a linear relation with coefficients presumably in
the Habiro ring $\calH_{\BZ[1/30]}$, in accordance with the idea that the rank of
$H^1_\dR(X)$ is $2$.

A final extension of our experiments involved adding one more variable $\lambda$
so that $X$ was now a fiber of a 1-parameter family of ellptic curves, parametrized
by $\lambda$. In that case, the $q$-hypergeometric origin of our elements in
$\calH^1_\nv(X)$ imply that they generate a $q$-holonomic
submodule of $\calH^1_\nv(X)$ whose corresponding linear $q$-difference equation
is a $6$th order deformation of the Picard-Fuchs equation of the family of elliptic
curves.

This $q$-holonomic $\BZ[\lambda,q]$-submodule in $\calH^1_\nv(X)$ of excess
rank $6$ is a $q$-deformation of the degree $2$ Picard-Fuchs equation of the family
of elliptic curves, and it is a very mysterious part in this new Habiro cohomology,
whose understanding may lead to a further refinement of this theory. 

Having understood the emerging structure, we proceed to formalize it
in the form of a definition, and then provably construct explicit
elements in the naive Habiro cohomology using two different methods: either
from a $q$-deformation of a hypergeometric motive in the form of a geometric
$D$-module, or by a push-forward of the Habiro ring of an \'etale map. The two
methods are complementary and roughly speaking mimic two approaches to compute
algebraic de Rham cohomology, namely either using $D$-modules
or using reduction theory on complexes of
algebraic forms (see e.g., Griffiths~\cite{Griffiths} and also the excellent
expositions of Cox-Katz~\cite[Sec.5.3]{Cox-Katz} and Lairez~\cite{Lairez}). 

We illustrate both methods with the Legendre family of elliptic curves, the
family of elliptic curves associated to the $4_1$ knot, and with the quintic 3-fold,
whose $q$-Picard-Fuchs equation already appeared in its genus $0$ quantum $K$-theory
~\cite{GS:quintic}.

\subsection{A definition}
\label{sub.habdef}

Our next task is to formalize the definition of the Habiro cohomology. This is
modelled on the definition of the Habiro ring $\calH_{R/\BZ[\lambda]}$ of
an \'etale $\BZ[\lambda]$-algebra $R$, defined in~\cite{GSWZ} and
presented in further detail by Scholze~\cite{Scholze:hab}. Let us recall this
definition here, which among other things motivates the forthcoming definition of
Habiro cohomology. 

We fix an \'etale $\BZ[\lambda]$-algebra $R$ (or synonymously an \'etale map
$\BZ[\lambda]\to R$ of rings), where $\lambda$ denotes a finite
and possibly empty set of variables. The ring $\BZ[\lambda]$ has the additional
structure of a $\Lambda$-ring with Adams operations $\psi_m(\lambda)=\lambda^m$
for $m \geq 1$, which induce $p$-Frobenius endomorphisms
on the $p$-adic completion $R^{\wedge}_p=\varprojlim_n R/(p^n) \cong R\otimes \BZ_p$
of $R$ for all but finitely many primes $p$, determined by
$\varphi_p(\lambda)=\lambda^p$. Two basic examples to keep in mind are the
\'etale algebras $R=\BZ[1/N]$ or $R=\BZ[\lambda,1/\Delta]$ for a single variable
$\lambda$.

In addition, we fix a collection of compatible complex roots of unity
$(\z_m)$ order $m$, that is a collection satisfying
\be
\label{zmdef}
\z_{mm'} \=\z_m\z_{m'}, \qquad (m,m')\=1, \qquad
(\z_{p^r})^p \=\z_{p^{r-1}}
\ee
for all positive coprime integers $m$ and $m'$, primes $p$ and positive integers $r$.

\begin{definition}
\label{def.habR}\cite{GSWZ}
The Habiro ring $\calH_{R}$ of an \'etale $\BZ[\lambda]$-algebra $R$   
consists of collections of series
\be
f(q)\=(f_m(q-\z_m))_{m \geq 1}
\in\prod_{m \geq 1}R_m[\![q-\z_m]\!]
\ee
where $R_m=R[\lambda^{1/m},\z_m]$ that satisfy the gluing 
\be
f_m(q-\z_{pm} + (\z_{pm} - \z_m))
\= \varphi_pf_{pm}(q-\z_{pm})\in (R_{pm})^\wedge_p[\![q-\z_{pm}]\!]
\ee
for all $m \geq 1$ and all primes $p$.  
\end{definition}

Fix an $R$-module $V$ with compatible Frobenius automorphism
\be
\varphi_p:V\otimes\BZ_p\to V\otimes\BZ_p
\ee
for all but finitely many primes $p$. For a positive integer $m$, we set
$V_m=V\otimes R_m$. 

\begin{definition}
\label{def.hab}  
The $\calH_{R}$-module $\calH(V)$ associated to $V$ consists of collections of series
\be
f(q)\=(f_m(q-\z_m))_{m \geq 1}
\in\prod_{m \geq 1}V_m\otimes\BQ[\z_m][\![q-\z_m]\!]
\ee
that are convergent for $|q-\z_m|_p<1$ and
satisfy the gluing 
\be
f_m(q-\z_{pm} + (\z_{pm} - \z_m)) \=
\varphi_pf_{pm}(q-\z_{pm})\in V_{pm}\otimes\BQ_p[\z_{mp}][\![q-\z_{pm}]\!]
\ee
for all $m \geq 1$ and all but finitely many primes $p$.  
\end{definition}

\begin{remark}
We will often consider collections on a subset of roots of unity, for example, the roots of unity with order $m$ coprime to some integer $N$.
\end{remark}

The series in $q-\z_m$ that we consider have logarithmically-growing denominators,
i.e., the the $p$-valuation of the coefficient of $(q-\z_m)^k$ is bounded below by
a multiple of $\log_pk$. This effectively comes from the fact that
$
\frac{\log q}{q-\z_m}
$
is a holomorphic function on $|q-\z_m|_p<1$ with logarithmically-growing denominators.
It is curious that the logarithmic-growth of the powers
of $q-\z_m$ matches the Christol-Dwork-Robba bounds for the $p$-adic valuation of
a fundamental solution (a series in $\lambda$ and $\log \lambda$)
of a $p$-adic differential equation that comes from geometry discussed in Kedlaya
~\cite[Thm.18.2.1,Thm.18.5.1]{Kedlaya:book}. 

The module $V$ that we consider will be the relative algebraic de Rham cohomology
$H^n_{\dR}(X/B)$ (modulo torsion) of a smooth proper map
\be
\label{bf}
f: X \to B, \qquad B=\Spec(R), \qquad R/\BZ[\lambda] \,\, \text{\'etale} \,.
\ee
This cohomology is equipped with a Frobenius endomorphism
for all but finitely many primes obtained from
the comparison to crystalline cohomology~\cite{Berthelot:comparison}.
(See~\cite{Kedlaya:practice,Kedlaya:effective} for an
elementary definition and effective computations of the Frobenius endomorphisms
using rigid cohomology.) Using this, we define
\be
\label{Hnaive}
\calH_\nv^{n}(X/B) \;:=\; \calH(H^n_{\dR}(X/B))\,.
\ee
We expect that this definition should capture some of the important features of
Habiro cohomology being developed by Wagner~\cite{Wagner:habiro}
and Scholze~\cite{Scholze:habcoh}.

\begin{remark}
\label{rem.GM}
An important feature of a smooth proper map of~\eqref{bf} is its Gauss--Manin
connection and the corresponding Picard-Fuchs equation, which will play a key
role in our methods to construct elements in $\calH_\nv^{n}(X/B)$.
\end{remark}

In the following sections we give two (geometric engineering) methods for
constructing elements in the Habiro cohomology. The first method introduced
in Section~\ref{sub.method1} uses as input a hypergeometric motive,
and was inspired among other things, by Shirai's work on a $q$-deformation of
the Legendre family of elliptic curves~\cite{Shirai}. The second 
method, a push-forward property from the Habiro ring of an \'etale map,
is discussed in Section~\ref{sub.method2}.

\subsection{Habiro cohomology classes from hypergeometric motives}
\label{sub.method1}

A rich source of proper smooth maps~\eqref{bf} come from hypergeometric
motives. The latter depend on combinatorial
data, akin to the parameters of classical hypergeometric functions. These motives
and their corresponding Hodge structures, point counts and $L$-functions are discussed
in detail in the excellent presentation of Roberts--Rodriguez-Villegas~\cite{RV:hyper}.

What's more, hypergeometric motives come equipped with an explicit Gauss--Manin
connection (a general feature of the maps in~\eqref{bf}) as well as an explicit
de Rham cohomology class and an explicit hypergeometric function, which both
satisfy the corresponding Picard-Fuchs equation.

The data that defines a hypergeometric motive is a pair of vectors
$\a=(\a_1,\dots,\a_{n+1})$ and $\b=(\b_1,\dots,\b_{n+1})$
of rational
numbers.
We will assume that
\begin{itemize}
\item[(A)]
the sets $\a$ and $\beta$ are disjoint with common denominator $N$,
and, for convenience, that $\b_{n+1}=1$ and $\a_j, \b_j \in (0,1]$ and
$\a_j<\b_j$ for all $j=1,\dots,n+1$. 
\end{itemize}

Given such a data $(\a,\b)$, one can define
\begin{itemize}
\item
  a smooth proper family $X/B$ of relative dimension $n$,
  where $B=\Spec(\BZ[\lambda,(\lambda(\lambda-1)N)^{-1}])$ (explicitly, take $X$
  to be the projectivisation of
  \be
  \{y,z\in\mathbb{A}^n,w,\lambda\in\mathbb{A}^1\;
  |\;y_i^N+z_i^N\=1,w^N+\lambda \, y_1^N\cdots y_n^N\=1\})\,,
  \ee
\item
  a hypergeometric power series
\be
\label{fab}
{}_{n+1}F_{n}(\a;\b;\lambda) \=  \sum_{k\in\BZ_{\geq0}}
  \frac{(\a_1)_k\cdots(\a_{n+1})_k}{(\b_1)_k\cdots(\b_{n+1})_k}
  \lambda^{k} \in \BQ[\![\lambda]\!]
  \ee
  (where $(a)_k=a(a+1)\dots (a+k-1)$ is the Pochhammer symbol) that satisfies the
  Picard-Fuchs equation
\be
\label{eq:de2}
P{\;}_{n+1}F_{n}(\a;\b;\lambda)=0, \qquad
P \= \th\prod_{j=1}^{n}(\th+\b_j-1)-\lambda\prod_{j=1}^{n+1}(\th+\a_j) \in \W
\ee
  where $\W \=\BZ[\lambda]\langle \th \rangle/(\th\lambda-\lambda\th-\lambda)$ and 
  $\th=\lambda \pt_\lambda$.
\item
  a class $\om=\om(\a;\b;\lambda) \in H_\dR^n(X/B)$ that is annihilated
  by the Gauss-Manin connection,
  i.e., satisfies $P \om =0$, explicitly, for $y_j^N=x_j$ we define
  \be
  \om(\a;\b;\lambda)\=(1-\lambda x_1\cdots x_{n})^{-\a_{n+1}}
  \bigwedge_{j=1}^{n} x_j^{\a_j}(1-x_j)^{\b_j-1-\a_j}\frac{dx_j}{x_j}\,.
  \ee
\end{itemize}

A detailed definition of the tuple $(X/B,P,\om)$, which of course depends on
$(\a,\b)$, is given in Section~\ref{sub.nfn}. We now explain how to use such
a tuple to define explicit classes in $\calH_\nv^{n}(X/B)$.

Remarkably, the definition uses essentially only the solution~\eqref{fab} of the
Picard-Fuchs equation, and is obtained in three steps
\be
\label{squigs}
{}_{n+1}F_{n}(\a;\b;\lambda) \rightsquigarrow
{}_{n+1}\phi_{n}(\a;\b;\lambda,q)
\rightsquigarrow (D_m)_{m \geq 1}
\rightsquigarrow
\omega_q \in \calH_\nv^{n}(X/B)\,,
\ee

The first step involves the $q$-deformation 
obtained by replacing the Pochhammer symbols in~\eqref{fab} by $q$-Pochhammer symbols
(where $(x;q)_k=(1-x)(1-qx)\dots(1-q^{k-1}x)$),
\be
\label{fabq}
{}_{n+1}\phi_{n}(\a;\b;\lambda,q) \=
\sum_{k\in\BZ_{\geq0}}
  \frac{(q^{\a_1};q)_k\cdots(q^{\a_{n+1}};q)_k}
  {(q^{\b_1};q)_k\cdots(q^{\b_{n+1}};q)_k}
  \lambda^{k} 
\ee
which is a (proper) $q$-hypergeometric series originally introduced by Heine.
The latter satisfies a linear $q$-difference equation (see ~\eqref{eq:deq} below).

The second step compares the solution ${}_{n+1}F_{n}(\a;\b;\lambda)$ and
its $q$-deformation ${}_{n+1}\phi_{n}(\a;\b;\lambda,q)$ near roots of unity
using a family $D_m$ of differential operators. To define them near $q=1$, first
we expand $(q^a;q)_k$ as a power series in
$(-1)^k(a)_k \BQ[k][\![q-1]\!]$ to obtain that 
\be
\label{Pk1}
  \prod_{i=1}^{n+1}\frac{(q^{\a_i};q)_k}{(q^{\b_i};q)_k}
  \= D_1(k,q-1)
  \prod_{i=1}^{n+1} \frac{(\a_i)_k}{(\b_i)_k},
  \qquad D_1(k,x) \in \BQ[k][\![q-1]\!]
\ee
with $D_1(k,0)=1$. Since $\th \lambda^k = k \lambda^k$ for all $k$, it follows
that
\be
\label{f1}
{\;}_{n+1}\phi_{n}(\a;\b;\lambda,q)
\= D_1(\th,q-1) {\;}_{n+1}F_{n}(\a;\b;\lambda)
, \qquad D_1(\th,q-1) \in \W_\BQ[\![q-1]\!]
\ee
where $\W_\BQ=\W \otimes \BQ$. More generally, we can expand $(q^a;q)_k$ around
$m$th roots of unity and obtain operators
$D_m(\lambda^{1/m}, \th, q-\z_m) \in \W_\BQ[\lambda^{1/m},\z_m][\![q-\z_m]\!]$
such that
\be
\label{fm}
{\;}_{n+1}\phi_{n}(\a;\b;\lambda^{1/m},q)
\= D_m(\lambda^{1/m}, \th, q-\z_m) {\;}_{n+1}F_{n}([\a]_m;[\b]_m;\lambda)
\ee
for all positive integers $m$ coprime to the common denominator of $\a$ and
$\b$, where for a rational number $a$ with denominator coprime to $m$, we denote
\be
\label{abracket}
[a]_m=\frac{a+\langle -a\rangle}{m} \in \BQ 
\ee
where $\langle b\rangle$ is the unique integer in $[0,m)\cap\BZ$ satisfying
$b\equiv\langle b\rangle\pmod{m}$.
Note that $[a]_m\in(0,1]$ with the same denominator as $a$.


To define our cohomology classes $[\om_{\a,\b}]$, we need an additional
normalisation factor, namely 
the collection of power series for each $a,b\in\BQ$ with
$a+b\neq0$ and $m$ coprime to the denominator of $a,b$ given by
\be
B_m(a,b;q-\z_m) \= \frac{\Gamma([a]_m+[b]_m)}{\Gamma([a]_m)\Gamma([b]_m)}
\frac{(q^{a+b};q)_\infty(q;q)_\infty}{(q^{a};q)_\infty(q^{b};q)_\infty}
  \in(q-\z_m)^{-1}\BQ[\z_m][\![q-\z_m]\!]\,.
\ee
Up to the important normalization factor involving the three values of the
$\Gamma$-function one can think of $B_m(a,b;q-\z_m)$ as the series expansion
of a $q$-Beta function at roots of unity.

With this we can define\footnote{The bracket $[a]_m$ for $a\in\BQ$ and
  $m\in\BZ_{>0}$ should not be confused with the bracket defining the de
  Rham cohomology class $[\om]$ for a form $\om$.}
\be
\label{omdef}
\begin{aligned}
\om_{m,q-\z_m} &\=
\prod_{j=1}^{n}
B_{m}(\a_j,\b_j-\a_j;q-\z_m)D_m(\lambda^{1/m},\th,q-\z_m)
[\om_{[\a]_m,[\b]_m}]\\
&\;\in\; H_\dR^{n}(X/B)\otimes \BQ[\lambda^{1/m},\z_m](q-\z_m)^{-n-1}[\![q-\z_m]\!] \,.
\end{aligned}
\ee
These series are generally convergent for $|q-\z_m|_p<1$ and seem to have
valuations decreasing logarithmically to minus infinity.

\begin{theorem}
\label{thm.1}
The collection $\om_q=(\om_{m,q-\z_m})_{m\geq1,(m,N)=1} \in \calH_\nv^{n}(X/B)$.
\end{theorem}

\begin{theorem}
\label{thm.1.2}
The form $\om_q$ generates a $q$-holonomic submodule of
$\calH_\nv^{n}(X/B)$ annihilated by the $q$-PF equation ~\eqref{eq:deq}.
\end{theorem}

The proof of these theorems is given in Section ~\ref{sec.method1} in three steps.
The first involves proving
that the power series $\om_{m,q-\z_m}$ is meromorphic on $|q|_p=1$ (with potential
poles at roots of unity of order at most $n+1$),
and hence can be re-expanded at $\z_{pm}$. This is by no means obvious since the
operators $\om_{m,q-\z_m}$ have factorially growing denominators. To achieve our
goal, we use a division of these operators by the Picard-Fuchs operator $P$, and
an explicit basis of solutions. The second step involves proving that the
expansions of $\om_{m,q-\z_m}$ and $\om_{m,q-\z_{pm}}$ agree, after a
Frobenius twist, which uses the tensor with $\BZ[\![\lambda]\!]$.
The second theorem follows form computations with the full basis of solutions in
$P$ showing that the $q$-difference operator multiplying the defining operators
has a right factor of $P$.
\begin{remark}

The proof of Theorem~\ref{thm.1} does not use the $q$-holonomicity of ${}_{n+1}\phi_n$
as a function of $\lambda$. Therefore, we could define a class in exactly the same
way we did when we started with ${}_{n+1}\phi_n$ by replacing it by
\be
  \sum_{k\in\BZ_{\geq0}}
  q^{k^{d}}\frac{(q^{\a_1};q)_k\cdots(q^{\a_{n+1}};q)_k}
  {(q^{\b_1};q)_k\cdots(q^{\b_{n+1}};q)_k}
  \lambda^{k}
\ee
for any $d\in\BZ$ for example. These examples for $k>2$ will not generally be
$q$-holonomic over $\BZ[q,\lambda]$.
\end{remark}


\begin{remark}
\label{rem.noPF}
Although the input data for the explicit classes in Theorem~\ref{thm.1} are geometric,
and hence the corresponding Picard-Fuchs operators are regular singular, this plays
no role in the proof of Theorem~\ref{thm.1}. Neither does the maximal unipotent
monodromy assumption on the Picard-Fuchs operator associated to a Calabi--Yau manifold. 
\end{remark}

\begin{remark}
\label{rem.hasse}
The constant terms
\be
  \om_{m,0}
  \=
  \sum_{k=0}^{p-1}
\prod_{j=1}^{n+1}
\frac{(\z_m^{\a_j};\z_m)_{k}}{(\z_m^{\b_j};\z_m)_{k}}
\lambda^{k/m}\om(\a;\b;\lambda)\,,
\ee
recover the Hasse polynomial modulo $p$ by the formula
\be
\Big(
\sum_{k=0}^{p-1}
\prod_{j=1}^{n+1}
\frac{(\z_p^{\a_j};\z_p)_{k}}{(\z_p^{\b_j};\z_p)_{k}}
\lambda^{k/p}\Big)^{p}
\;\equiv\;
\sum_{k=0}^{p-1}
\prod_{j=1}^{n+1}
\frac{(\a_j)_{k}}{(\b_j)_{k}}
\lambda^{k}\pmod{p}\,.
\ee
\end{remark}

\begin{remark}
\label{rem.factorials}
A special case of hypergeometric motives that depend on two vectors
\newline
$a=(a_1,\dots,a_K)$ and $b=(b_1,\dots,b_L)$ of positive integers that satisfy the
balancing condition
\be
\sum_{i=1}^K a_i = \sum_{j=1}^L b_j=n+1
\ee
was considered in~\cite{RV:hyper}. The corresponding hypergeometric series whose
summand is a ratio of products of factorials is given by
\be
\label{Fla}
f_{a,b}(\lambda) \=
\sum_{k=0}^{\infty} \frac{\prod_{i=1}^{K} (a_ik)!}{\prod_{j=1}^{L} (b_j k)!}
  \lambda^k
\=
{}_{n+1}F_{n}
(\a;\b;C \lambda)
\in \BZ[\![\lambda]\!]
\,,
\ee
and
\be
\label{a2alpha}
\a\=\Big(\frac{1}{a_1},\frac{2}{a_1},\cdots,\frac{a_1}{a_1},\cdots\Big)\,,\quad
  \b\=\Big(\frac{1}{b_1},\frac{2}{b_1},\cdots,\frac{b_1}{b_1},\cdots\Big)\,,\quad
  C\=\frac{\prod_ia_i^{a_i}}{\prod_jb_j^{b_j}}.
\ee
The $q$-hypergeometric series is
\be
\label{Flaq}
f_{a,b}(\lambda,q) \=
\sum_{k=0}^{\infty} \frac{\prod_i (q;q)_{a_i k}}{\prod_j (q;q)_{b_j k}} \lambda^k
\in \BZ[q^{\pm 1}][\![\lambda]\!] \,,
\ee

The identity
\be
  \prod_{j=1}^{N-1}\Gamma(\tfrac{j}{N})
  \=
  \sqrt{\frac{(2\pi)^{N-1}}{N}}
\ee
implies that 
\be
  \prod_{j=1}^{N-1}\frac{(q^{j/N};q)_\infty}{(q;q)_\infty}
\ee
are units in the Habiro ring of the quadratic field $\BQ(\sqrt{N})$
(see remark~\ref{rem:gam.exp.per} below).
Hence, in the examples of factorial ratios, $\om$ can be defined without the
need of the normalisation in Equation~\eqref{omdef}.

%
\end{remark}


\subsection{Elements in the Habiro ring of an \'etale $\BZ[t]$-algebra}
\label{sub.exphab}

A second method to define explicit elements in the Habiro cohomology of smooth
maps $X/B$ will be via a push-forward from the Habiro ring $\calH_{R/\BZ[t]}$
of an \'etale $\BZ[t]$-algebra $R$. In this section we give two constructions of
explicit elements of the said Habiro ring.
Both constructions have hypergeometric origin, and the first one produces elements
in the Habiro ring $\calH_{R/\BZ[\lambda]}$ for $R$ a localisation of $\BZ[\lambda]$
(hence the corresponding power series for $q$ near a root of unit have coefficients
rational functions), whereas the second for $R$ a finite \'etale of a localisation
of $\BZ[\lambda]$ (hence the corresponding power series for $q$ near a root of unit
have coefficients algebraic functions). 

The first construction starts with a rational function
\be
\label{fcalA}
f_{\calA,Q}(t) \= \frac{1}{1- \sum_{\a \in \calA} (-1)^{Q_{\a\a}}t_\a} \= \sum_{k \in \BZ_{\geq 0}^\calA}(-1)^{\diag(Q)k}
\frac{(\sum_\a k_\a)!}{\prod_\a (k_\a)!} \prod_\a t_\a^{k_\a}
\ee
in a finite set $\calA$ of variables $t_\a$, where $t=(t_\a)_{\a \in \calA}$,
$k=(k_\a)_{\a \in \calA}$ and $Q$ is an integral symetric matrix (which is currently all redundant expect for the diagonal modulo two). Replacing the factorials with $q$-factorials and adding a quadratic from, we obtain
a $q$-deformation $f(t,q)$ of $f(t)$ defined by
\be
\label{fcalAq}
f_{\calA,Q}(t,q) \=  \sum_{k \in \BZ_{\geq 0}^\calA}
(-1)^{\diag(Q)k}q^{\frac{1}{2}k^{t}Qk+\frac{1}{2}\diag(Q)k}
\frac{(q;q)_{\sum_\a k_a}}{\prod_\a (q;q)_{k_\a}} \prod_\a t_\a^{k_\a} \,.
\ee

\begin{theorem}
\label{thm.ct2}
For every finite set $\calA$ and $Q$ integral symmetric matrix indexed by $\calA$, we have 
$f_{\calA,Q}(t,q) \in \calH_{\BZ[t][f_\calA(t)]/\BZ[t]}$.
\end{theorem}

We can then define a map $\Lambda$-rings
\be
\label{2lambda}
\BZ[t] \to \BZ[x^{\pm1},\lambda], \qquad
t_\a \mapsto x^\a \lambda:= ( \prod_{i=1}^{n+1} x_i^{\a_i} )\lambda
\ee
where $\a=(\a_1,\dots,\a_{n+1})$ and $x^\a=x_1^{\a_1} \dots x_{n+1}^{\a_{n+1}}$
which sends $f_\calA(t)$ to $(1-\lambda h_\calA(x))^{-1}$ where
\be
\label{hA}
h_{\calA,Q}(x) \= \sum_{\a \in \calA}(-1)^{Q_{\a\a}}x^\a \in \BZ[x^{\pm 1}] 
\ee
and sends $f_{\calA,Q}(t,q)$ to
\be
f_{\calA,Q}(x,\lambda,q) \;=\!\!  \sum_{k \in \BZ_{\geq 0}^\calA}
(-1)^{\diag(Q)k}q^{\frac{1}{2}k^{t}Qk+\frac{1}{2}\diag(Q)k}
\frac{(q;q)_{\sum_\a k_a}}{\prod_\a (q;q)_{k_\a}} \prod_\a x_\a^{k_\a \a}
\lambda^{\sum_\a k_\a} \in \BZ[q,x^{\pm1}][\![\lambda]\!] \,.
\ee

Theorem~\ref{thm.ct2} and the functoriality of the Habiro ring under the map
~\eqref{2lambda} imply the following:

\begin{theorem}
\label{thm.ct}
The collection
$f_{\calA,Q}(x,\lambda,q) \in
\calH_{\BZ[x,\lambda][\tfrac{1}{1-\lambda h_{\calA,Q}(x)}]/\BZ[x,\lambda]}$.
\end{theorem}


The second construction depends on a symmetric $N \times N$ matrix $A$ with integer
entries. The constructed elements are symmetrised Nahm sums, and our aim is to
give a self-contained definition that they are explcit elements of the
corresponding Habiro ring avoiding the use of formal Gaussian
integration and the residue formulas with respect to auxillary variables as was done
in~\cite[Sec.4.1]{GSWZ}. 

Recall the Nahm sum $f_A(t,q)$ associated to a symmetric, integral $N \times N$
matrix $A$

\be
\label{fAdef}
f_A(t,q) \= \sum_{n \in \BZ^N_{\ge0}} \frac{
(-1)^{\mathrm{diag}(A)\cdot n}\,
q^{\frac{1}{2}(n^t A n + \mathrm{diag}(A)\cdot n)} }{(q;q)_{n_1}
\dots (q;q)_{n_N}} t_1^{n_1} \dots t_N^{n_N} \in \BQ(q)[\![t]\!]
\ee
where $n=(n_1,\dots,n_N)$ and $t=(t_1,\dots,t_N)$.

The matrix $A$ defines system of (so-called Nahm) equations
\be
\label{XA}
1-z_i - (-1)^{A_{ii}}t_i \prod_{j=1}^{N}z_j^{A_{ij}} \= 0, \qquad i=1,\dots, N 
\ee
which we abbreviate by $1-z-(-1)^A t z^A$. This in turn defines an \'etale map
\be
\label{RA}
\BZ[t] \to R_A, \qquad 
\BZ[t]=\BZ[t_1,\dots,t_N], \qquad
R_A:= \BZ[t][z_1^{\pm 1},\dots,z_N^{\pm 1},\tfrac{1}{\delta}]/(1-z-(-1)^A t z^A)
\ee
where
\be
\label{deltadef}
\delta(t) \= \det\big(\diag(1-z)\,A+\diag(z)\big) \prod_{j=1}^{N}z_{j}^{-A_{jj}}\,.
\ee

Consider the symmetrisation
\be
\label{SfA}
\Sf_A(t,q) \= f_A(t,q) f_A(t,q^{-1}) 
\ee
of $f_A(t,q)$. Our goal is to give a self-contained proof of the following
theorem implicit from~\cite{GSWZ}.

\begin{theorem}[\cite{GSWZ}]
\label{thm.fA1} 
For every symmetric, integer $N \times N$ matrix $A$, we have:
\be
\label{Sf}
\Sf_A(t,q) \in \calH_{R_A/\BZ[t]} \,.
\ee
\end{theorem}

It was known that $\BZ[q^{\pm 1}][\![t]\!]$ as a consequence of the admissibility
of the series $f_A(t,q)$~\cite{GSWZ}. The proof of the above theorem
uses a manifestly integral formula for the symmetrisation $\Sf_A(t,q)$,
which is new and gives a self-contained proof of Theorem~\ref{thm.fA1}. 

\begin{proposition}
\label{prop.A1}
For every symmetric, integer $N \times N$ matrix $A$, we have:
\be
\label{SfA1}
\Sf_A(t,q) \= 
  \sum_{k\in\BZ_{\geq0}^N}
  (-1)^{\diag(A-I)k}q^{-k(A-I)k/2-\diag(A-I)k/2}\prod_{j=1}^N
\left[\!\!\begin{array}{cc} (Ak)_j+A_{jj}-1 \\ k_j \end{array}\!\!\right]_q
t_1^{k_1} \dots t_N^{k_N}.
\ee
\end{proposition}




\begin{remark}
\label{rem.Aetale}  
The \'etale map $\BZ[t] \to R_A$ is not finite, but the one
$\BZ[t,\tfrac{1}{\delta}] \to R_A$ is. Note also that there is a natural
embedding $R_A \hookrightarrow \BZ[\![t]\!]$, which plays an important
role in~\cite{GSWZ}.
\end{remark}

\begin{remark}
\label{rem.twist}
Twisting of $q$-hypergeometric elements $\Sf_A(t,q)$ gives elements in the
Habiro ring $\calH_{R/\BZ[t]}$. Indeed, given a symmetric
matrix $B$ with integer entries, we can twist $\Sf_A(t,q)$ by $B+A-I$ to obtain
(up to the signs of $t$, which is a probably regrettable choice of convention)
further elements in the Habiro ring
\be
\sum_{k\in\BZ_{\geq0}^N}
(-1)^{\diag(B)k}q^{-kBk/2-\diag(B)k/2}\prod_{j=1}^N
\left[\!\!\begin{array}{cc} (Ak)_j+A_{jj}-1 \\ k_j \end{array}\!\!\right]_q
t_1^{k_1} \dots t_N^{k_N} \in \calH_{R_A/\BZ[t]} \,.
\ee
These elements no longer come from a symmetrization associated to a matrix with
integer entries. 
\end{remark}

\begin{question}
Does the action for a symmetric matrix $B$,
\be
t_1^{k_1} \dots t_N^{k_N}
\mapsto
(-1)^{\diag(B)k}q^{-k^tBk/2-\diag(B)k/2}
t_1^{k_1} \dots t_N^{k_N}\,,
\ee
(known as the $q$-Borel transform) descend to the Habiro ring of $R_A/\BZ[t]$
or even to the Habiro cohomology for families over $t$? 
\end{question}


\subsection{Habiro cohomology classes from push-forward}
\label{sub.method2}

In this section we give a second method to construct explicit
classes in Habiro cohomology, defined as push-forwards of elements of the Habiro
ring of an \'etale map. To do so, we fix an \'etale map
\be
\BZ[x,\lambda] \to R
\ee
where $x=(x_1,\dots,x_N)$ and $\lambda=(\lambda_1,\dots,\lambda_L)$, and use the
$p$-Frobenius action on differential forms on $\Spec(\BZ[x,\lambda])$ given by 
\be
\label{frobs2}
\vphi_p(a) \= a^p, \qquad \vphi_p(da) \= a^{p-1}da 
\ee
for $a \in \{x_1,\dots,x_N,\lambda_1,\dots,\lambda_L\}$.

Let $\calH_{R/\BZ[x,\lambda]}^{\mathrm{an}}$ denote the subring of elements
of $\calH_{R/\BZ[x,\lambda]}$ whose $(q-\z_m)$-series, with coefficients in $R_m$,
have positive radius of convergence with respect to the completion of any prime ideal.


Setting $X=\Spec(R)$, consider the push-forward $X/B$ of the map
$X \to \Spec(\BZ[x,\lambda])$ along $x$, where
$B=\Spec(\BZ[\lambda,\tfrac{1}{\Delta(\lambda)}])$.
Fix a form $\om\in\Om^{N}(X/B)$ invariant under the
action of the $p$-Frobenius for all but finitely many $p$. For example,
\be
\label{omx}
\om \= \pi^*\wedge_{i=1}^N \frac{dx_i}{x_i}, \qquad
\pi: X \to \Spec(\BZ[x,\lambda]) \,.
\ee

Then we have the following.

\begin{theorem}
\label{thm.push}
There is a push-forward map
\be
\label{push}
\calH_{R/\BZ[x,\lambda]}^{\mathrm{an}} \to \calH^N_\nv(X/B),
\qquad f \mapsto [f \om]\,.
\ee 
\end{theorem}

Roughly, speaking, we have
\be
\label{roughly}
\text{$x$-push forward} \;\approx\; \Res_{x=0} \;\approx\; \text{$x$-constant term} \,.
\ee
A word of caution: we will work exclusively with affine varieties and explicit
(integral) de Rham cohomology classes, and not touch upon hypercohomology, Hodge
cohomology of smooth projective curves.

\subsection{Habiro cohomology classes of Calabi--Yau manifolds}
\label{sub.CY}

An application of Theorems~\ref{thm.ct} and~\ref{thm.push} constructs
explicit Habiro cohomology classes on Calabi--Yau families over an one-dimensional
space. 

Fix a Laurent polynomial $h_\calA(x)$ as in Equation~\eqref{hA} and
Theorem~\ref{thm.ct}, and let $X$ denote the 0-locus of $1-\lambda h_\calA(x)=0$
and $X/B$ the corresponding smooth map for
$B=\Spec(\BZ[\lambda,\tfrac{1}{\Delta(\lambda)}])$. Let
\be
\om \=  \bigwedge_{i=1}^{n+1} \frac{d x_i}{x_i} \,.
\ee

\begin{theorem}
\label{thm.pfh}
We have
\be
[\Res_{1-\lambda h_\calA(x)}(f_\calA(x,\lambda,q) \om)] \in \calH_\nv^n(X/B) \,.
\ee
\end{theorem}

Note that
\be
\label{hA1}
\Res_{x=0} \frac{1}{1-h_\calA(x)} \om 
\= \sum_{k=0}^\infty \tr (h(x)^k) \lambda^k 
\ee
where $\tr(g(x)) \= Res_{x=0} g(x) \om$ denotes the constant term (i.e., the
coefficient of $x^0$) of a Laurent polynomial $g(x)$.

One might think that the restriction on all coefficients of $h_\calA(x)$ to be 1 in
Theorem~\ref{thm.pfh} is too restrictive. Yet,  
the class of smooth maps $X/B$ that come from polynomials $h$ with all coefficients
$1$ includes all known (and conjecturally all) Calabi--Yau manifolds with a complex
1-dimensional Kalher moduli space. A database for those is available in~\cite{AESZ};
see also~\cite{Candelas:local}. 
In particular, it includes the generating series of the famous Apery sequence
\be
\label{apery}
a_n \= \sum_{k=0}^n \binom{n}{k}^2 \binom{n+k}{k}^2 \= \tr(h_A(x,y,z)^n)
\ee
with
\be
h_A(x_1,x_2,x_3) \= \frac{(x_1+x_2)(x_3+1)(x_1+x_2+x_3)(x_2+x_3+1)}{x_1 x_2 x_3} 
\ee
see e.g.,~\cite[Rem.1.4]{Straub}. 

The sequence $(a_n)$ satisfies a second order recursion with coefficients polynomials
in $n$ of degree $3$, and consequently its generating series satisfies a linear
differential equation of order $3$ with coefficients polynomials in $\lambda$ of
degree $2$. On the other hand, the \texttt{HolonomicFunctions} package gives
that $a_n(q)$ satisfies a linear $q$-difference equation of order $6$ with coefficients
polynomials in $(q^n,q)$ of degree $(58,188)$. There is no guarantee that
these are the smallest parameters but a search with precomputed data and smaller
set of parameters failed to produce a simpler equation. 


The two methods of constructing elements of the Habiro cohomology
from Sections~\ref{sub.method1} and ~\ref{sub.method2} are closely related.
On one hand, one can work with the Picard-Fuchs
equations and their quantisations or work with the original forms that solve the
Picard-Fuchs equation as cohomology classes.
Both lead to slowly growing denominators for what seem to be different reasons.
The first begins with large denominators that disappear after the reduction to
a remainder in the Weyl algebra. The second begins as $p$-integral with generally
large denominators coming from other ideals, which turn into slowly growing
$p$-denominators after reduction of forms. At the end both methods lead to the
same slowly growing denominators.

\subsection{Future directions}
\label{sub.future}

In this section we give a conjectural lift of the perturbative invariants
of one-cusped oriented (finite volume, complete) hyperbolic 3-manifolds
(such as hyperbolic knot complements) to elements of geometrically-constructed
Habiro rings and Habiro cohomologies. Each such manifold $M$ has a variety
of conjugacy classes of
representations of its fundamental group in $\SL_2(\BC)$. Given a pair of peripheral
elements (meridian and longitude) that generate the integral homology of $\pt M$,
there a map from the above variety to $(\BC^*)^2$ that records the
eigenvalues of the meridian and the longitude. The image of the character variety
of $M$ under this map is an affine curve $X_M$ in $(\BC^*)^2$ with defining
polynomial $A_K(x,y) \in \BZ[x^{\pm 1},y^{\pm 1}]$, the so-called $A$-polynomial of
$M$~\cite{CCGLS}. This defines an \'etale map
\be
\label{ketale}
\BZ[x^{\pm 1}] \to
R_M\=\BZ[x^{\pm 1}][y^{\pm 1},\Delta(x)^{-1}]/(A_M(x,y)) \,,
\ee
where $\Delta(x)$ is the discriminant of $A_M$ with respect to $x$. 
Perturbative Complex Chern--Simons theory (that is, the collection of
power series of~\cite{DG,DG2} at roots of unity, with $x$-deformation included)
ought to define a section of a line bundle over the Habiro ring
$\calH_{R_M/\BZ[x^{\pm1}]}$, and the 3D-index of Dimofte--Gaoitto--Gukov~\cite{DGG1}
ought to define an element of the Habiro ring $\calH_{R_M/\BZ[x^{\pm1}]}$.
A full construction of the conjectured elements of the above-mentioned Habiro ring
yet available, nonetheless, the topological invariance of the above section
(as a power series in $q-1$) is known~\cite{GSW}, and so is the topological
invariance of the 3D-index; see~\cite{GHRS,GK:mero} and also the discussion
in~\cite{GW:periods}.

Despite the conjectural tone of the above discussion, all numerical observations
presented in the introduction are provably true for the simplest hyperbolic $4_1$
knot, and we explain this in Section~\ref{sec.41} below.


\subsection*{Acknowledgements} 

The authors wish to thank Don Zagier and especially Peter Scholze and Ferdinand
Wagner for many enlightening conversations. C.W. has been supported by the Huawei
Young Talents Program at Institut des Hautes Études Scientifiques, France. The
authors wish to thank the Max-Planck Institute for Mathematics and the I.H.E.S.
for their hospitality. 


\section{From hypergeometric motives to Habiro cohomology}
\label{sec.method1}

In this section we present in detail a method to construct Habiro cohomology
classes on a hypergeometric motive. To do so, we recall some elementary lemmas
concerning the expansions of $q$-Pochhammer symbols near complex roots of unity,
and then discuss how to obtain elements in the Habiro cohomology of the Fermat
curves. Then, we recall the basic properties of hypergeometric motives, and their
associated functions, and use them to define the collections of series that
appear in Theorem~\ref{thm.1}. Combining all this together, we prove the
convergence and gluing properties of Theorem~\ref{thm.1}.
In the remaining part of this section, we discuss in detail several examples.



\subsection{$q$-Pochhammer symbols near roots of unity}
\label{sub.asy}

In this section present some elementary lemmas regarding the asymptotics of
the $q$-Pochhammer symbols near roots of unity. Note that finite $q$-Pochhammer
symbols can be expressed as quotients of two infinite $q$-Pochhammer symbols
as follows
\be
\label{qpk}
(q^a;q)_k=\tfrac{(q^a;q)_\infty}{(q^{a+k};q)_\infty}
\ee
for every $k \in \BZ_{\geq 0}$ and $a \in \BQ_{>0}$.
\begin{lemma}
\label{lem.q1}
For every rational number $a\in\BQ$, and with $q=e^\hbar$, as $\hbar\to0^{-}$ we
have:  
\be
\label{eq:qpinf}
  (q^a;q)_\infty
  \;\sim\;
  \sqrt{2\pi}\frac{(-\hbar)^{\frac{1}{2}-a}}{\Gamma(a)}
  \exp\Big(\frac{\pi^2}{6\hbar}-\sum_{\ell=2}^\infty
  B_{\ell}(a)
  \frac{B_{\ell-1}(1)}{\ell-1}\frac{\hbar^{\ell-1}}{\ell!}
  \Big)\,.
\ee
\end{lemma}

\begin{proof}
This follows from standard summation methods such as the Euler-MacLaurin or the
Abel-Plana summation methods when $\alpha$ is not a non-positive integer
and follows for non-positive integers as both side are identically equal to zero.
\end{proof}

Lemma~\ref{lem.q1} and Equation~\eqref{qpk} imply that
for every $k \in \BZ_{\geq 0}$ and $a\in\BQ_{>0}$, and with $q=e^\hbar$, we have:  
\be
  \label{eq:qp.q1}
\begin{aligned}
  (q^a;q)_k
  &\=
  (-1)^k(a)_k\hbar^k
  \exp\Big(\sum_{\ell=2}^{\infty}(B_{\ell}(k+a)-B_{\ell}(a))
  \frac{B_{\ell-1}(1)}{\ell-1}\frac{\hbar^{\ell-1}}{\ell!}\Big)\\
  &\;\in\;
  (-1)^k(a)_k\hbar^a\BQ[a,k][\![\hbar]\!]\,.
\end{aligned}
\ee

More generally, we will need the expansion of $(q^a;q)_\infty$ near roots of
unity $\z_m$. This is given in the next lemma, using the notation of Equation
~\eqref{abracket}.

\begin{lemma}
\label{lem:qpzm}
For every rational number $a\in\BQ$, integer $m \geq 1$ coprime to the
  denominator of $a$ and $q=\z_m e^\hbar$, as $\hbar\to0^{-}$ we have
\be
\label{eq:qpinfm}
\begin{aligned}
  &(q^a;q)_\infty
  \;\sim\;
  \sqrt{2\pi}\frac{(-m\hbar)^{\frac{1}{2}-[a]_m}}{\Gamma([a]_m)}
  \prod_{\langle-a\rangle_m\neq j=0}^{m-1}
  (1-\z_m^{a+j})^{\frac{1}{2}-\frac{a+j}{m}}
  \exp\Big(\frac{\pi^2}{6m\hbar}\\
  &-\sum_{\ell=2}^\infty
  B_{\ell}([a]_m)
  \frac{B_{\ell-1}(1)}{\ell-1}\frac{(m\hbar)^{\ell-1}}{\ell!}+\!\!\!\!\!\!
  \sum_{\langle-a\rangle_m\neq j=0}^{m-1}
  \sum_{\ell=2}^{\infty}\frac{B_{\ell}(\frac{a+j}{m})}{\ell!}
  \Li_{2-\ell}(\z_m^{a+j})(m\hbar)^{\ell-1}\Big) \,.
\end{aligned}
\ee
\end{lemma}

\begin{proof}
Notice that
\be
  (q^a;q)_\infty
  \=
  \prod_{j=0}^{m-1}(q^{a+j};q^m)_\infty\,.
\ee
This then reduces the computation to Equation~\eqref{eq:qpinf} and the asymptotic
identity as $\hbar\to0^{-}$
\be
  (q^ax;q)_\infty
  \;\sim\;
  \exp\Big(\sum_{\ell=0}^{\infty}\frac{B_{\ell}(a)}{\ell!}
  \Li_{2-\ell}(x)\hbar^{\ell-1}\Big)\,,\qquad q\=e^{\hbar}\,,
\ee
which again follows from standard summation methods.
\end{proof}

Note that Lemma~\ref{lem:qpzm} and Equation~\eqref{qpk} imply an expansion of
$(q^a;q)_k$ near roots of unity of order coprime to the denominator of $a$. 

The above asymptotic expansions motivate the collection of series
for some fixed $a\in\BQ\setminus\BZ$, $m\in\BZ_{>0}$ coprime to the denominator
of $a$
\be
  \Gamma_m(a;q-\z_m)
  \=
  \Gamma([a]_m)\frac{(q^{a};q)_\infty}{(q;q)_\infty}
  \in(q-\z_m)^{-[a]_m}\BQ[\z_m][\![q-\z_m]\!]
\ee
which, in a sense, is an expansion of the $q$-$\Gamma$ function (originally
introduced by Jackson~\cite{Jackson}) at roots of unity.
The next remark relates this function to exponential periods.

\begin{remark}
\label{rem:gam.exp.per}
The collection
\be
  \om_{m,q-\z_m}
  \=
  \Gamma_m(a;q-\z_m)[x^{-[a]_m}dx]_{-x}
\ee
(where the subscript $_{-x}$ denotes twisted de Rham cohomology) 
should give an element of the Habiro cohomology associated to the twisted
de Rham complex with differential $d-dx\wedge$ associated to the exponential
periods (see~\cite[Sec. 4.3]{KZ:periods})
\be
  \int x^{-a}e^{-x}dx\,.
\ee
\end{remark}

\subsection{The Fermat curve}
\label{sec:fermat}

The first example that we will discuss will be the smooth projective
Fermat curve $X_N$ over $\BZ[1/N]$ whose affine equation given by
\be
\label{fermatdef}
y^N+z^N=1 
\ee
for a positive integer $N>0$. It is a curve of genus of genus $(N-1)(N-2)/2$.
This is a classic example discussed in detail among other places,
in~\cite{Coleman:fermat} and~\cite{Gross:periods}.
For $\a,\b\in\frac{1}{N}\BZ_{>0}$ we have de Rham cohomology classes
associated to the 1-forms (where $y^N=x$, thus $z^N=1-x$)
\be
\label{eq:omab}
  \om_{\a,\b}
  \=
  x^{\a-1}(1-x)^{\b-1}dx\,.
\ee
The periods of these forms are expressed in terms of Euler's $B$-function
\be
  B(\a,\b)
  \=
  \frac{\Gamma(\a)\Gamma(\b)}{\Gamma(\a+\b)}
  \=
  \int_{0}^{1}\om_{\a,\b}\,.
\ee
Notice that the functional equations for the $\Gamma$ function descend to the
cohomology classes, for example,
\be
  d(x^{\a}(1-x)^{\b})
  \=
  \a\om_{\a,\b}-(\a+\b)\om_{\a+1,\b}
  \=
  -\b\om_{\a,\b}+(\a+\b)\om_{\a,\b+1}\,.
\ee

Coleman~\cite{Coleman:fermat} computed the Frobenius matrix for $H_\dR^1(X_N)$
in terms of the basis of holomorphic one-forms indexed by
$\a,\b\in(0,1)\cap\frac{1}{N}\BZ$ with $\a+\b\neq1$ and given by
\be
  v_{\a,\b}
  \=(\a+\b-\lfloor \a+\b\rfloor)^{\lfloor \a+\b\rfloor}
  [\om_{\a,\b}]\,.
\ee
He showed~\cite[Prop.1.4, Thm.1.7]{Coleman:fermat}
that the action of the Frobenius on $v_{\a,\b}$ is expressed in terms of the
$p$-adic $\Gamma$-function (sometimes called the Morita $\Gamma$-function
~\cite{Morita,Boyarsky}) by 
\be
  \varphi_p v_{[\a]_p,[\b]_p}
  \=
  (-1)^{\lfloor \a+\b\rfloor}
  p^{\lfloor 2-[\a]_p-[\b]_p\rfloor}
  \frac{\Gamma_p(\a+\b-\lfloor\a+\b\rfloor)}{\Gamma_p(\a)\Gamma_p(\b)}
  v_{\a,\b}\,,
\ee
with $[a]_m$ given by Equation~\eqref{abracket}.
If $p>2N$ is prime, this can be expressed in the basis
$\om$ by the formula
\be
\label{eq:colegam}
  \varphi_p [\om_{[\a]_p,[\b]_p}]
  \=
  p^{[\a+\b]_{p}+1-[\a]_{p}-[\b]_{p}}
  \frac{\Gamma_p(\a+\b)}{\Gamma_p(\a)\Gamma_p(\b)}
  \frac{\Gamma([\a+\b]_{p})}
  {\Gamma([\a]_{p}+[\b]_{p})}
  [\om_{\a,\b}]\,.
\ee
This follows form the fact that, with our lower bound on $p$,
$\a+\b-\lfloor\a+\b\rfloor$ does not contain a factor of $p$ and that
\be
  1>2-[\a]_p-[\b]_p-([\a+\b]_{p}+1-[\a]_{p}-[\b]_{p})\geq 1/p
\ee
is also an element of $\frac{1}{N}\BZ$ and hence is at least $0$.


One can associate a $q$-deformation of this form from the infinite Pochhammer
symbol. Indeed, for each $m$ and $q=\z_me^{\hbar}$, we define
\be
\label{eq:Bmin}
  B_m(\a,\b;q-\z_m)
  \=
  \frac{\Gamma([\a]_m+[\b]_m)}{\Gamma([\a]_m)\Gamma([\b]_m)} 
  \frac{(q^{\a+\b};q)_\infty(q;q)_\infty}
  {(q^{\a};q)_\infty(q^{\b};q)_\infty}
  \in(q-\z_m)^{-1}\BQ[\z_m][\![q-\z_m]\!] \,.
\ee
To see that the left hand side of the above equation lies in
$(q-\z_m)^{-1}\BQ[\z_m][\![q-\z_m]\!]$ and to compute the corresponding power
series expansion, we use Equation~\eqref{eq:qpinfm} which implies that with
$q=\z_m e^\hbar$, we have
\be
\begin{aligned}
  &B_m(\a,\b;q-\z_m)
  \=
  \frac{\Gamma([\a]_m+[\b]_m)}{\Gamma([\a+\b]_m)}
  (-m\hbar)^{[\a]_m+[\b]_m-[\a+\b]_m-1}
  \prod_{j=1}^{m-1}(1-\z_m^{j})^{\frac{1}{2}-\frac{j}{m}}\\
  &\times\!\!\!\!\prod_{m[\a+\b]_m-\a-\b\neq j=0}^{m-1}
  \!\!\!\!\!\!\!(1-\z_m^{\a+\b+j})^{\frac{1}{2}-\frac{\a+\b+j}{m}}
  \!\!\!\!\prod_{m[\a]_m-\a\neq j=0}^{m-1}
  \!\!\!\!\!\!\!(1-\z_m^{\a+j})^{-\frac{1}{2}+\frac{\a+j}{m}}
  \!\!\!\!\prod_{m[\b]_m-\b\neq j=0}^{m-1}
  \!\!\!\!\!\!\!(1-\z_m^{\b+j})^{-\frac{1}{2}+\frac{\b+j}{m}}\\
  &\times\exp\Big(-\sum_{\ell=2}^\infty
  (B_{\ell}([\a+\b]_m)+B_{\ell}(1)-B_{\ell}([\a]_m)-B_{\ell}([\b]_m))
  \frac{B_{\ell-1}(1)}{\ell-1}\frac{(m\hbar)^{\ell-1}}{\ell!}\\
  &+
  \sum_{\ell=2}^{\infty}
  \Big(\sum_{m[\a+\b]_m-\a-\b\neq j=0}^{m-1}
  \frac{B_{\ell}(\frac{\a+\b+j}{m})}{\ell!}\Li_{2-\ell}(\z_m^{\a+\b+j})
  +\sum_{j=1}^{m-1}\frac{B_{\ell}(\frac{j}{m})}{\ell!}\Li_{2-\ell}(\z_m^{j})\\
  &-\sum_{m[\a]_m-\a\neq j=0}^{m-1}
  \frac{B_{\ell}(\frac{\a+j}{m})}{\ell!}\Li_{2-\ell}(\z_m^{\a+j})
  -\sum_{m[\b]_m-\b\neq j=0}^{m-1}
  \frac{B_{\ell}(\frac{\b+j}{m})}{\ell!}\Li_{2-\ell}(\z_m^{\b+j})\Big)
  (m\hbar)^{\ell-1}\Big)\,.
\end{aligned}
\ee
This implies the previous inclusion~\eqref{eq:Bmin} since,
$[\a]_m+[\b]_m-[\a+\b]_m-1 \in \{0,-1\}$ and
the $m$-th power of the four $j$-products is equal to
\be
\prod_{j=1}^{m-1}(1-\z_m^{j})^{\a+\b+\langle
  -\a-\b\rangle_m-\a-\langle-\a\rangle_m-\b-\langle-\b\rangle_m}
  \=
  1\pmod{(\BQ[\z_m]^{\times})^m}\,.
\ee

We are interested in the $p$-adic properties of the collection
$(B_m(\a,\b;q-\z_m))_{m \geq 1, (m,N)=1}$. This is described in the
following theorem, which again involves the $p$-adic $\Gamma$-function
and illustrates that the $q$-Pochhammer symbol ``knows'' about the
$p$-adic $\Gamma$-function.

\begin{theorem}
\label{thm.beta}
Fix $\a,\b\in\frac{1}{N}\BZ$ and prime $p$ with $(p,N)=1$.
Then the series
\be
\log(q)B_m(\a,\b;q-\z_m)
\ee
is convergent for $|q-\z_m|_p<1$ and satisfies
\be
\begin{aligned}
  B_m(\a,\b;q-\z_m)
&\= 
  p^{[[\a]_m+[\b]_m]_{p}+1-[\a]_{pm}-[\b]_{pm}}
  \frac{\Gamma_p([\a]_m+[\b]_m)}{\Gamma_p([\a]_m)\Gamma_p([\b]_m)}
  \frac{\Gamma([[\a]_m+[\b]_m]_{p})}
  {\Gamma([\a]_{pm}+[\b]_{pm})} \\ 
&  \qquad\qquad \times B_{pm}(\a,\b;q-\z_{pm})\\
&\;\in\;
  (q-\z_{pm})^{-1}\BQ_p[\z_{pm}][\![q-\z_{pm}]\!]\,.
\end{aligned}
\ee
\end{theorem}

\begin{proof}
Suppose that $\a,\b\in\BZ_{\geq0}$, then
\be
  \frac{(q^{\a+\b};q)_\infty(q;q)_\infty}{(q^{\a};q)_\infty(q^{\b};q)_\infty}
  \in\BZ[(1-q^j)^{-1}\;|\;j\in\BZ_{>0}]\,.
\ee
Therefore, $B_m(\a,\b;q-\z_m)$ is a rational function with at
most simple poles at roots of unity.
It follows that $\log(q)B_m(\a,\b;q-\z_m)$ is holomorphic for $|q|_p=1$.
Notice that the coefficients in the expansion of $\log(q)B(\a,\b;q-\z_m)$ are
continuous and therefore we see that, since $\frac{1}{N}\BZ\subseteq\BZ_p$, we must
also have $\log(q)B_m(\a,\b;q-\z_m)$ holomorphic for $\a,\b\in\frac{1}{N}\BZ$.
This completes the proof of the first statement.

For the second recall that for integer $n,\ell$ with $0\leq\ell<p$ we have
\be
  \Gamma_p(pn-\ell)
  \=
  (-1)^{pn-\ell}\frac{\Gamma(pn-\ell)}{p^{n-1}\Gamma(n)}\,.
\ee
Notice that $[\a]_{m}=
p[\a]_{pm}-\ell$ for some $\ell\in[0,p)\cap\BZ$.
Therefore, for integral $\a$, we find that
\be
  \frac{\Gamma([\a]_m)}{\Gamma([\a]_{pm})}
  \=
  \Gamma_p([\a]_m)(-1)^{[\a]_m}p^{[\a]_{pm}-1}\,.
\ee
Therefore, for all integral $\a,\b$ we find that
\be
\begin{aligned}
  &\frac{B_{pm}(\a,\b;q-\z_{pm})}{B_m(\a,\b;q-\z_m)}
  \=
  \frac{\Gamma([\a]_m)\Gamma([\b]_m)}
  {\Gamma([\a]_m+[\b]_m)}
  \frac{\Gamma([\a]_{pm}+[\b]_{pm})}{\Gamma([\a]_{pm})\Gamma([\b]_{pm})}\\
  &\=
  p^{[\a]_{pm}+[\b]_{pm}-[[\a]_m+[\b]_m]_{p}-1}
  \frac{\Gamma_p([\a]_m)\Gamma_p([\b]_m)}
  {\Gamma_p([\a]_m+[\b]_m)}\frac{\Gamma([\a]_{pm}+[\b]_{pm})}{
    \Gamma([[\a]_m+[\b]_m]_{p})}
  \in\BQ[\![q-\z_{pm}]\!]\,.
\end{aligned}
\ee
Again the density of $\BZ_{>0}$ in $\BZ_p$ and the continuity of the coefficients
in the $q-\z_{pm}$ expansion implies this equality for all
$\a,\b\in\frac{1}{N}\BZ$.
\end{proof}

Consider the collection of series
\be
  g_{\a,\b,m}(q-\z_m)
  \=
  B_m(\a,\b;q-\z_m)
  [\om_{[\a]_m,[\b]_m}]
  \in H_{\dR}^{1}(X_N/\BZ[1/(2N)!])\otimes\BQ[\![q-\z_m]\!]
\ee
for $m \geq 1$ with $(m,N)=1$.   
The previous theorem and Coleman's formula~\eqref{eq:colegam} implies the following:

\begin{corollary}
\label{cor.fermat}  
For all $\a,\b \in (0,1)\cap\tfrac{1}{N}\BZ$ with $\a+\b\neq1$,
the collection
$g_{\a,\b}(q)$ is an element of the Habiro cohomology of
the Fermat curve $X_N$, i.e.,
\be
  g_{\a,\b}(q)\in\calH^1_\nv(X_N/\BZ[1/(2N)!])\,.
\ee
\end{corollary}

\subsection{Basics on balanced hypergeometric series}
\label{sub.nfn}

In this section, we recall some of the basic properties of the hypergeometric
function ${}_{n+1}F_{n}$ that are well-known and may be found for example in
Gasper--Rahman~\cite{Gasper}. Recall that for $\a\in\BQ^{n+1}$ and
$\b\in\BQ^{n+1}$ with $\b_{n+1}=1$, we have
\be
  {}_{n+1}F_n(\a;\b;\lambda)
  \=
  \sum_{k\in\BZ_{\geq0}}
  \frac{(\a_1)_k\cdots(\a_{n+1})_k}{(\b_1)_k\cdots(\b_{n})_kk!}
  \lambda^{k}\,.
\ee
We assume that $\a_i$ and $\b_j$ are pairwise disjoint. 
For $\b_j - \a_j \not\in \BZ_{\leq 0}$ and $\b_j > \a_j$ (see \cite{RV:hyper}),
these holomorphic functions near $\lambda=0$ can be described by
the integral formula
\be
  {}_{n+1}F_n(\a;\b;\lambda)
  \=
  \Big(\prod_{j=1}^{n}\frac{\Gamma(\b_j)}{\Gamma(\a_j)\Gamma(\b_j-\a_j)}\Big)
  \int_0^1\!\!\cdots\!\!\int_0^1
  \om(\a;\b;\lambda)
\ee
where
\be
\label{omab}
\om(\a;\b;\lambda)
  \=
  (1-\lambda x_1\cdots x_{n})^{-\a_{n+1}}
  \prod_{j=1}^{n}
  x_j^{\a_j}(1-x_j)^{\b_j-1-\a_j}\frac{dx_j}{x_j}\,.
\ee
For example, for $\lambda$ near $0$, we have
\be
  {}_2F_1(1/2,1/2;1;\lambda)
  \=
  \frac{1}{\pi}
  \int_0^1\frac{dx}{\sqrt{x(1-x)(1-\lambda x)}}
\ee

These functions satisfy the elementary differential equations
\be
\label{eq:de}
P{\;}_{n+1}F_{n}(\a;\b;\lambda)=0, \qquad
P \= \th\prod_{j=1}^{n}(\th+\b_j-1)-\lambda\prod_{j=1}^{n+1}(\th+\a_j) \in \W\,.
\ee
We can construct a fundamental basis of solutions by applying Frobenius's method.
To phrase it, let
$\overline{\b}=(\overline{\b}_1,\dots, \overline{\b}_{n+1}) \in \BQ/\BZ$
denote the $\BQ/\BZ$-reduction of $\b$ and $\b_i^\max=\max\{\b_j \,\, |
\overline{\b}_j=\overline{\b}_i \}$ and $d_i$ be the cardinality of the
set $\{\b_j \,\, | \overline{\b}_j=\overline{\b}_i \}$.

The solutions are all of the form
\be
  {}_{n+1}F_n(\a+\rho;\b+\rho;\lambda)\lambda^{\rho}\,,
\ee
where we take $\rho=1-\b_i^\max+\epsilon$ and consider the coefficients of the
$\epsilon$ expansion to the power $\mathrm{O}(\epsilon^{d_i})$. 

When all $\overline{\b}_i\in\BQ/\BZ$ are distinct, the multiplicities
$d_i=1$ and therefore the solutions are all given for $j=1,\dots,n$ by
\be
  {}_{n+1}F_n(\a+1-\b_j;\b+1-\b_j;\lambda)\lambda^{1-\b_j}\,.
\ee

By considering the exact forms for $i=1,\dots,n$ given by
\be
  d\Big(
  \frac{\Gamma(\b_i)}{\Gamma(\a_i)\Gamma(\b_i-\a_i)}
  \frac{x_i^{\a_i}(1-x_i)^{\b_i-1-\a_i}}
  {(1-\lambda x_1\cdots x_{n})^{\a_{n+1}}}
  \prod_{i\neq j=1}^n
  \frac{\Gamma(\b_j)}{\Gamma(\a_j)\Gamma(\b_j-\a_j)}
  x_j^{\a_j}(1-x_j)^{\b_j-1-\a_j}\frac{dx_j}{x_j}\Big),
\ee
one can show that $P\om(\a;\b;\lambda)$ is an exact form. This implies that
the integral over any closed contour will produce a solution to the differential
equation.

There is a natural $q$-deformation of these hypergeometric sums originally
introduced by Heine, given by 
\be
  {}_{n+1}\phi_{n}(\a;\b;\lambda,q)
  \=
  \sum_{k\in\BZ_{\geq0}}
  \frac{(q^{\a_1};q)_k\cdots(q^{\a_{n+1}};q)_k}
  {(q^{\b_1};q)_k\cdots(q^{\b_{n}};q)_k(q;q)_k}
  \lambda^{k}\,.
\ee
These functions satisfy the $q$-difference equation
\be
\label{eq:deq}
\Big((1-\sigma)\prod_{j=1}^n(1-q^{\b_i-1}\sigma)
-\lambda\prod_{j=1}^{n+1}(1-q^{\a_i}\sigma)\Big)\;
  {}_{n+1}\phi_{n}(\a;\b;\lambda,q)
  \=0\,.
\ee
(These functions were shown to be quantum modular forms in~\cite{GW:qmod}.)
By Frobenius's method applied to the $q$-difference equations, one can construct
an analogous basis of solutions
\be
  {}_{n+1}\phi_n(\a+\rho;\b+\rho;\lambda,q)\lambda^{\rho}
\ee
again taking $\rho=1-\b_i^\max+\epsilon$ and considering the coefficients
of the $\epsilon$ expansion to the power $\mathrm{O}(\epsilon^{d_i})$.

Using~\eqref{eq:qp.q1}, we can compare the $q\to1$ limit of $\phi$ with $F$ 
as follows
\be
  {}_{n+1}\phi_n(\a+\rho;\b+\rho;\lambda,q)\lambda^{\rho}
  \=
  {}_{n+1}F_n(\a+\rho;\b+\rho;\lambda)\lambda^{\rho}
  +\mathrm{O}(q-1)\,.
\ee

These hypergeometric function have a simple relation when one shifts entries
of $\a$ or $\b$ by $1$. Explicitly if $\delta_i$ is the vector with $i$th
coordinate $1$ and all other coordinates zero, we have
\be
\label{eq:fashift}
\begin{aligned}
  \a_i\;{}_{n+1}F_{n}(\a+\delta_i;\b;\lambda)
  & \=
  (\a_i+\th)\;{}_{n+1}F_{n}(\a;\b;\lambda)\,,
\\
  (1-q^{\a_i})\;{}_{n+1}\phi_{n}(\a+\delta_i;\b;\lambda,q)
  & \=
  (1-q^{\a_i}\sigma)\;{}_{n+1}\phi_{n}(\a;\b;\lambda,q)\,.
\end{aligned}
\ee
Therefore, without loss of generality, we will assume
that $\a,\b\in(0,1]^{n+1}$.

\subsection{From $q$-deformations to $D$-modules}

We are interested in the full $q-1$ series expansion of the $q$-deformation
${}_{n+1}\phi_n$. In fact, we want the $q-1$ expansion of a fundamental basis
of solutions to the $q$-Picard-Fuchs equation obtained by the Frobenius method.
A convenient basis to work in for these computations is given by
\be
\label{eq:fnorm}
  {\;}_{n+1}\widetilde\phi_n(\a+\rho;\b+\rho;\lambda,q)\lambda^{\rho} \= 
  C_\rho(\a;\b,q) 
  {\;}_{n+1}\phi_n(\a+\rho;\b+\rho;\lambda,q)\lambda^{\rho}\,.
\ee
where
\be
\label{Cdef}
  C_\rho(\a;\b,q) \=
  \prod_{j=1}^{n+1}
  \frac{(q^{\b_j+\rho};q)_\infty(q^{\a_j};q)_\infty}{
    (q^{\a_j+\rho};q)_\infty(q^{\b_j};q)_\infty} \,.
\ee
Equation~\eqref{eq:qpinf} implies that 
\be
C_\rho(\a;\b,q) \in
\prod_{j=1}^{n+1}\frac{\Gamma(\a_j+\rho)\Gamma(\b_j)}{
  \Gamma(\a_j)\Gamma(\b_j+\rho)} \BQ[\![q-1]\!]
\ee
and further implies that there exists $D_{1}(y,q-1)\in\BQ[y][\![q-1]\!]$ so that 
\be
{\;}_{n+1}\widetilde\phi_n(\a+\rho;\b+\rho;\lambda,q)\lambda^{\rho}
 \= D_{1}(\th,q-1) \prod_{j=1}^{n+1}
  \frac{\Gamma(\a_j+\rho)\Gamma(\b_j)}{\Gamma(\a_j)\Gamma(\b_j+\rho)}
  {\;}_{n+1}F_n(\a+\rho;\b+\rho;\lambda)\lambda^{\rho}\,.
\ee
We are also interested in the expansions when $q$ is near $\z_m$.
To understand this recall $[a]_m$ from Equation~\eqref{abracket}.
Note that if $a\in(0,1]$ then $[a]_m\in(0,1]$ with the same denominator as $a$.
Then we consider the congruence sums for $i=0,\dots,m-1$
\be
\label{eq:phinorm}
  \prod_{j=1}^{n+1}
  \frac{(q^{\b_j+i+m[\rho]_m};q)_\infty(q^{\a_j};q)_\infty
  }{(q^{\a_j+i+m[\rho]_m};q)_\infty(q^{\b_j};q)_\infty
  }
  \sum_{k\in\BZ_{\geq0}}
  \frac{(q^{\a_1+i+m[\rho]_m};q)_{mk}\cdots(q^{\a_{n+1}+i+m[\rho]_m};q)_{mk}}
  {(q^{\b_1+i+m[\rho]_m};q)_{mk}\cdots(q^{\b_{n+1}+i+m[\rho]_m};q)_{mk}}
  \lambda^{k+\rho}\,.
\ee
Using Equation~\eqref{eq:qpinfm} of Lemma~\ref{lem:qpzm}, we find that for
$q$ near $\z_m$, there exists
$D_{m,i}(y,q-\z_m)\in\BQ[y,\z_m][\![q-\z_m]\!]$ so that
\be
\begin{aligned}
  &\prod_{j=1}^{n+1}\frac{(q^{\b_j+i+m[\rho]_m};q)_\infty(q^{\a_j};q)_\infty}
  {(q^{\a_j+i+m[\rho]_m};q)_\infty(q^{\b_j};q)_\infty}
  \frac{(q^{\a_j+i+m[\rho]_m};q)_{mk}}
  {(q^{\b_j+i+m[\rho]_m};q)_{mk}}\\
  &=
  \prod_{j=1}^{n+1}\frac{\Gamma([\a_j+i]_m+[\rho]_m+k)\Gamma([\b_j]_m)}
  {\Gamma([\b_j+i]_m+[\rho]_m+k)\Gamma([\a_j]_m)}
  D_{m,i}(\tfrac{mk+i+m[\rho]_m}{m},q-\z_m)
\end{aligned}
\ee
Therefore, we see that~\eqref{eq:phinorm} is equal to
\be\label{eq:dmj}
\begin{aligned}
  D_{m,i}(\tfrac{1}{m}\th,q-\z_m)
  \sum_{k\in\BZ}
  \prod_{j=1}^{n+1}\Big(\frac{\Gamma([\a_j+i]_m+[\rho]_m+k)\Gamma([\b_j]_m)}
  {\Gamma([\b_j+i]_m+[\rho]_m+k)\Gamma([\a_j]_m)}\Big)
  \lambda^{mk+i+m[\rho]_m} \,.
\end{aligned}
\ee
Altogether, we find the following:
\begin{lemma}
\label{lem:Dexists}
There exists $D_m(y,\lambda,q-\z_m)\in\BQ[\th,\lambda,\z_m][\![q-\z_m]\!]$
such that
\be
\begin{aligned}
  & {\;}_{n+1}\widetilde\phi_n(\a+\rho;\b+\rho;\lambda,q)\lambda^{\rho}\\
  &\=
  D_{m}(\tfrac{1}{m}\th,\lambda,q-\z_m)
  \prod_{j=1}^{n+1}\frac{\Gamma([\a_j]_m+[\rho]_m)\Gamma([\b_j]_m)}
  {\Gamma([\b_j]_m+[\rho]_m)\Gamma([\a_j]_m)}
  \sum_{k\in\BZ}
  \prod_{j=1}^{n+1}\Big(\frac{([\a_j]_m+[\rho]_m)_k}
  {([\b_j]_m+[\rho]_m)_k}\Big)
  \lambda^{mk+m[\rho]_m} \,.
\end{aligned}
\ee
\end{lemma}

\begin{proof}
We can always increase the arguments of the $\Gamma$ functions in the
numerator and decrease their arguments in the denominator by applying the
relations~\eqref{eq:fashift}. Therefore, for every $i=0,\dots,m-1$,
there exists $S_i\in\W$ so that
\be
\begin{aligned}
  &\sum_{k\in\BZ}
  \prod_{j=1}^{n+1}\frac{\Gamma([\a_j+i]_m+[\rho]_m+k)\Gamma([\b_j]_m)}
  {\Gamma([\b_j+i]_m+[\rho]_m+k)\Gamma([\a_j]_m)}
  \lambda^{mk+i+m[\rho]_m}\\
  &\!\!=
  \lambda^i
  S_i
  \sum_{k\in\BZ}
  \prod_{j=1}^{n+1}\frac{\Gamma([\a_j]_m+[\rho]_m+k)\Gamma([\b_j]_m)}
  {\Gamma([\b_j]_m+[\rho]_m+k)\Gamma([\a_j]_m)}
  \lambda^{mk+m[\rho]_m}\!.
\end{aligned}
\ee
Therefore, we define
\be
  D_m(\tfrac{1}{m}\th,\lambda,q-\z_m)
  \=
  \sum_{i=0}^{m-1}
  D_{m,i}(\tfrac{1}{m}\th,q-\z_m)\lambda^iS_i
\ee
The proof then follows from Equation~\eqref{eq:dmj}.
\end{proof}

From its very definition we see that the action of
$D_m(\tfrac{1}{m}\th,\lambda,q-\z_m)$ on the classical hypergeometric
functions satisfies the $q$-Picard-Fuchs equation~\eqref{eq:deq} for all solutions
to the classical Picard-Fuchs $P_m\in\W$,
where
\be
  P_m
  \=
  \th\prod_{j=1}^{n}(\frac{1}{m}\th+[\b_j]_m-1)-\lambda^m\prod_{j=1}^{n+1}(\frac{1}{m}\th+[\a_j]_m)
\ee
so that for all $\rho=1-\b_i^\max+\epsilon+\mathrm{O}(\epsilon^{d_i})$
\be
  P_m\sum_{k\in\BZ}
  \prod_{j=1}^{n+1}\frac{\Gamma([\a_j]_m+[\rho]_m+k)\Gamma([\b_j]_m)}
  {\Gamma([\b_j]_m+[\rho]_m+k)\Gamma([\a_j]_m)}\lambda^{mk+m[\rho]_m}\=0\,.
\ee
This implies the following:

\begin{lemma}
\label{lem:qdif}
We have
\be
\Big((1-\sigma)\prod_{j=1}^n(1-q^{\b_i-1}\sigma)
-\lambda\prod_{j=1}^{n+1}(1-q^{\a_i}\sigma)\Big)D_{m}(\tfrac{1}{m}\th,\lambda,q-\z_m)
\= Q_mP_m
\ee
for some operator $Q_m\in\W_{\BQ}[\z_m][\![q-\z_m]\!]$.
\end{lemma}

\subsection{Convergence and gluing}
\label{sub.conv}

The operator $D_{m}(\th,\lambda,q-\z_m)$ can be written as
\be
\label{DDrem}
D_{m}(\th,\lambda,q-\z_m)
\= QP + D_{m}^{\rem}(\th,\lambda,q-\z_m)
\ee
where $D^\rem$ denotes the remainder of $D$ divided on the right by the Picard-Fuchs
operator $P$ and $n+1$ is the order of $P$. 
In this section, we will show that
the series in $q-\z_m$ is convergent once multiplied by $\log(q)^{n+1}$.
Recall that $N$ denotes the common denominator of $\a$, $\b$
and the function $C_\rho(\a;\b,q)$ from~\eqref{Cdef}. Using Equation
~\eqref{eq:qpinfm}, we can expand it into power series in $q-\z_m$ for all $m \geq 1$
with $(m,N)=1$:
\be
\label{Cseries}
C_\rho(\a;\b,q) \in
\prod_{j=1}^{n+1}\frac{\Gamma([\a_j+\rho]_m)\Gamma([\b_j]_m)}
  {\Gamma([\b_j+\rho]_m)\Gamma([\a_j]_m)} \BQ[\![q-\z_m]\!] \,.
\ee

\begin{lemma}
\label{lem:anal}
For $\rho\in\BZ_p$, the above expansions of
\be
\label{eq:hol.lem}
\log(q)^{n+1}
C_\rho(\a;\b,q)
\prod_{j=1}^{n+1}\frac{\Gamma([\b_j+\rho]_m)\Gamma([\a_j]_m)}{\Gamma([\a_j+\rho]_m)
  \Gamma([\b_j]_m)}
\ee
define holomorphic functions for all $|q-\z_m|_p<1$ with $(m,N)=(p,N)=1$.
\end{lemma}

\begin{proof}
For $\rho\in\BZ_{\geq0}$ we see that
\be
  \frac{(q^{\b_1+\rho};q)_\infty\cdots(q^{\b_{n+1}+\rho};q)_\infty}
  {(q^{\a_1+\rho};q)_\infty\cdots(q^{\a_{n+1}+\rho};q)_\infty}
  \frac{(q^{\a_1};q)_\infty\cdots(q^{\a_{n+1}};q)_\infty}{
    (q^{\b_1};q)_\infty\cdots(q^{\b_{n+1}};q)_\infty}
  \=
  \frac{(q^{\a_1};q)_\rho\cdots(q^{\a_{n+1}};q)_\rho}{
    (q^{\b_1};q)_\rho\cdots(q^{\b_{n+1}};q)_\rho}
    \in\BQ(q^{1/N})\,.
\ee
This function has potential poles at roots of unity and we see that the order
of the poles is bounded above by $n+1$ from Lemma~\ref{lem:qpzm}.
Therefore, we see that Equation~\eqref{eq:hol.lem} gives a holomorphic function
on the set $|q|_p=1$.
Notice that the coefficients, up to some global factor, are continuous functions
of $\rho$.
Therefore, we see that their growth is determined by the behaviour on $\BZ_{\geq0}$,
since this is dense in $\BZ_p$.
\end{proof}

\begin{theorem}
\label{thm:bounds.on.denom}
The series in $q-\z_m$
\be
  \log(q)^{n+1}D_{m}^{\rem}(\th,\lambda,q-\z_m)
\ee
are convergent for $|q-\z_m|_p<1$.
\end{theorem}

\begin{proof}
Since $D_m^{\rem}$ is the remainder by right division with an $(n+1)$-th order
operator $P$, we We can write
\be
\label{DDrem2}
D_{m}^{\rem}(\th,\lambda,q-\z_m) 
\= \sum_{i=0}^{n} D_{m}^{\rem,i}(\lambda,q-\z_m)\th^{i} \,.
\ee
Let $U(\lambda)$ be the fundamental basis of solutions to Equation~\eqref{eq:de}.
Then taking the corresponding solutions $\Phi$ to the $q$-difference equation
from Equation~\eqref{eq:fnorm}, we find that
\be
\log(q)^{n+1}(D_{m}^{\rem,0}(\lambda,q-\z_m),\dots,D_{m}^{\rem,n}(\lambda,q-\z_m))
U(\lambda)
  \=
  \log(q)^{n+1}\Phi(\lambda,q)\,.
\ee
By lemma~\ref{lem:anal}, the RHS is convergent for all $|q-\z_m|_p<1$ and therefore,
by simply multiplying both sides by $U(\lambda)^{-1}$ we obtain that
$\log(q)^{n+1}D_{m}^{\rem,i}(\lambda,q-\z_m)$ is convergent for $|q-\z_m|_p<1$.
\end{proof}

Therefore, we see that the operators $\log(q)^{n+1}D_{m}^{\rem,i}(\lambda,q-\z_m)$ can
be re-expanded at other roots of unity.
However, Lemma~\ref{lem:Dexists} implies that the arguments
of the classical function ${}_{n+1}F_n$ depend on the order of the root of unity
$m$ and are given by $([\a]_m,[\b]_m)$. However, varying $m$ generates a finite set 
$\{([\a]_m,[\b]_m)\}_{m\in\BZ_{>0}}$.

We can consider the $D$-module generated by ${}_{n+1}F_n$ with
$\a,\b\in\Big((0,1]\cap \frac{1}{N}\BZ\Big)^{n+1}$.
From Frobenius's algorithm, we see that this module contains the solution space.
The naive $q$-deformation is compatible wih this space and we obtain the solution
space to the $q$-Picard-Fuchs likewise.
With this full $D$-module and $\sigma$-module, we can therefore, find operators
$D_{m}(\th,\lambda,q-\z_m)$
that act on the $D$-module to produce the $q-\z_m$ expansions of the element of
the $\sigma$-module.
These operators can be re-expanded to series in $q-\z_{pm}$ once we $p$-adically
complete from Theorem~\ref{thm:bounds.on.denom}.

With Theorem~\ref{thm:bounds.on.denom} in hand we can finally prove Theorem
~\ref{thm.1}.

\begin{proof}[Proof of Theorem~\ref{thm.1}]
Consider the smooth projective variety $X/B$
defined by the affine equations
\be
\begin{aligned}
  w^N\=1-\lambda y_1^N\cdots y^N_n\,,\quad
  y_j^N+z_j^N\=1, \quad j=1,\dots,n 
\end{aligned}
\ee
and with $B=\Spec(\BZ[\lambda,(\lambda(\lambda-1)N)^{-1}])$.
It is easy to see that if $U$ is the affine variety defined by the above equations
then we have $Z=X\setminus U$ is a divisor with normal crossings.

It follows that with $x_j=y_j^N, \, 1-x_j=z_j^N$ for $j=1,\dots,n$, we have
\be
  \om(\a;\b;\lambda)
  \=
  (1-\lambda x_1 \dots x_n)^{-\a_{n+1}}
  \bigwedge_{j=1}^{n}x_j^{\a_j}(1-x_j)^{\b_j-1-\a_j}\frac{dx_j}{x_j}
  \in H^n_\dR(U/B) \,.
\ee
It is easy to see that the residue map $H_\dR^n(U/B) \to H_\dR^{n-1}(Z/B)$
vanishes, hence by the long exact sequence for de Rham cohomology for the pair
$(X,U)$, the above classes lift to well-defined elements in $H^n_\dR(X/B)$. 

The D-modules generated by
$\{\om([\a]_m;[\b]_m;\lambda)\}_{m\in\BZ_{>0}}$ are submodules
of $H^n_\dR(X/B)$. The operators $D_{m}(\th,\lambda,q-\z_m)$ act on these submodules.
For each $m$, we have (as defined in Equation~\eqref{omdef})
\be
  \om_{m,q-\z_m}
  \=
  \prod_{j=1}^{n}
  B_{m}(\a_j,\b_j-\a_j;q-\z_m)
  D_{m}(\th,\lambda^{\frac{1}{m}},q-\z_{m})
  \om([\a]_m,[\b]_m,\lambda^{\frac{1}{m}})\,.
\ee
We will now use the specialization to $\lambda=0$ to complete the argument.  
Since $X/B$ is proper smooth, it follows by the finiteness of coherent cohomology
(as was kindly communicated to us by P. Scholze and F. Wagner)
that $H_{\dR}^{n}(X/B)$ is a
finitely-generated $\BZ[\lambda,(\lambda(\lambda-1)N)^{-1}]$-module. Modulo torsion,
(which amounts to enlarging $\BZ[\lambda],(\lambda(\lambda-1)N)^{-1}]$ to
$\BZ[\lambda,\Delta(\lambda)^{-1}]$ for some suitably chosen $\Delta(\lambda)
\in \BZ[\lambda]$, a Noetherian ring) $H_{\dR}^{n}(X/B)$ is finite, projective
hence a submodule of a free $\BZ[\lambda,\Delta(\lambda)^{-1}]$-module.
It follows that modulo torsion
$H_{\dR}^{n}(X/B)$ injects into $H_{\dR}^{n}(X/B)\otimes\BQ(\!(\lambda)\!)$.
Under base-change $\widehat{X}=X\times_B\widehat{B}$ of $X$ to
$\widehat{B}=\mathrm{Spec}(\BQ(\!(\lambda)\!))$, 
we obtain the maps
\be
\label{eq:hom.maps}
  H_{\dR}^{n}(X/B)\otimes\BQ
  \hookrightarrow
  H_{\dR}^{n}(X/B)\otimes\BQ(\!(\lambda)\!)
  \cong
  H_{\dR}^{n}(\widehat{X}/\widehat{B})
  \cong H_{\dR}^{n}(X_N^n\times\widehat{B}/\widehat{B})\,,
\ee
where $X_N$ is the Fermat curve from Section~\ref{sec:fermat}.
The last isomorphism follows from the identity
\be
  (1-\lambda x_1\cdots x_n)^{\frac{1}{N}}
  \=
  \sum_{k=0}^{\infty}(-1)^k\binom{\frac{1}{N}}{k}x_1^k\cdots x_n^k\lambda^k\,.
\ee
The Frobenius morphisms are compatible with the maps~\eqref{eq:hom.maps}.

Under these maps, $\om(\a;\b;\lambda)\in H_{\dR}^{n}(X/B)$ is sent to the class
\be
  {}_{n+1}F_{n}(\alpha;\beta;\lambda)
  \bigwedge_{j=1}^{n}[\om_{\a_j,\b_j-\a_j}]
  \in H_{\dR}^{n}(X_N^n\times\widehat{B}/\widehat{B})\,,
\ee
where $\om_{\a,\b}$ are the classes defined in Equation~\eqref{eq:omab}.
Therefore, we see that in this $\lambda$-completed ring we have
\be
\begin{aligned}
  \om_{m,q-\z_m}
  &\=
  {}_{n+1}\phi_{n}(\alpha;\beta;\lambda^{\frac{1}{m}},\z_m+q-\z_m)
  \bigwedge_{j=1}^{n}\om_{[\a_j]_m,[\b_j-\a_j]_m}\\
  &\;\in\;
  H_{\dR}^{n}(X_N^n\times\widehat{B}/\widehat{B})[\lambda^{\frac{1}{m}},\z_m][\![q-\z_m]\!]\,.
\end{aligned}
\ee
The gluing then follows from the $B$-normalisation, which accounts for the action
of Frobenius on the forms $\om_{a,b}$ (as considered in Section~\ref{sec:fermat}),
and the fact that ${}_{n+1}\phi_{n}$ is given by the same functions at each root
of unity once we account for the action of Frobenius, which takes
$\lambda^{\frac{1}{pm}}\mapsto\lambda^{\frac{1}{m}}$. Since the first map
in Equation~\eqref{eq:hom.maps} is an injection, we see that an equality in
$H_{\dR}^{n}(X/B)\otimes\BQ_p(\!(\lambda)\!)$ proves the equality in
$H_{\dR}^{n}(X/B)\otimes\BQ_p$, and this concludes the proof of the theorem.
\end{proof}

In the remaining of this section, we discuss three examples of explicit elements
of the Habiro cohomology illustrating Theorem~\ref{thm.1}. 

\subsection{The Legendre family of elliptic curves}
\label{sub.leg1a}

Our first example is the Legendre family of elliptic curves $E/B$
defined by 
\be
  y^2\=x(x-1)(x-16\lambda)
\ee
where $\lambda \neq 0, 1/16$. We can think of the above equation as defining
a family of affine elliptic curves $E_\lambda$ parametrized by
$\lambda \in B=\Spec(\BZ[\lambda,\tfrac{1}{\lambda(\lambda-1)}])$.
The holomorphic one-form is given by $\om=dx/y$.
It satisfies the differential equation
\be
P(\lambda,\th)\om \=0 
\ee  
where $\th=\lambda \pt_\lambda$ as before and 
\be
  P(\lambda,\th)
  \=
  (16\lambda - 1)\th^2 + 16\lambda\th + 4\lambda
  \=
  4\lambda(2\th+1)^2-\th^2\,.
\ee
Indeed,
\be
  P(\lambda,\th)\om
  \=
  d(-32x^2(x-1)^2\lambda y^{-3})
  \=0\in H^{1}(E/B)\,.
\ee
We can solve the PF equation $P(\lambda,\th)f(\lambda)$ for a power series in
$\lambda$ and find the unique solution 
\be
\label{PFll}
  P(\lambda,\th)f(\lambda)\=0, \qquad  f(\lambda)
  \=
  \sum_{k=0}^{\infty}
  \frac{(2k)!^2}{k!^4}\lambda^k\in\BZ[\![\lambda]\!], \quad f(0)\=1 \,. 
\ee
The naive $q$-deformation\footnote{This is different from the
  one considered in Shirai~\cite{Shirai}. See the following section Shirai's
  $q$-deformation.}
\be
  f(\lambda,q)
  \=
  \sum_{k=0}^{\infty}
  \frac{(q;q)_{2k}^2}{(q;q)_{k}^4}\lambda^k
\ee
satisfies the $q$-difference equation
\be
\label{eq:qPF.v1}
\begin{aligned}
  &(1-\lambda)f(\lambda,q)-2(1+q\lambda)f(q\lambda,q)
  +(1+(2q-q^2)\lambda)f(q^2\lambda,q) + 4q^2\lambda f(q^3\lambda,q)\\
  &\qquad- (-2q^3 + q^2)\lambda f(q^4\lambda,q) - 2q^3\lambda f(q^5\lambda,q)
  - q^4\lambda f(q^6\lambda,q)
  \=0\,.
\end{aligned}
\ee
This equation is a $q$-deformation of the Picard-Fuchs equation. Indeed,
substituting $q=e^\hbar$ and expanding into power series in $\hbar$, one sees that
the coefficient of $\hbar$ is zero, and that of $\hbar^2$ is the Picard-Fuchs
equation~\eqref{PFll} for $f(\lambda)$. 

The series $f(\lambda,q)$ can be expanded around each complex root of unity.  
For example, the expansion at $q=1$ is given by
\be
\label{DDleg1}
f(\lambda,q) \=
\big(1 + \th^2(q-1)
+\tfrac{1}{12}(6\th^4+2\th^3-5\th^2)(q-1)^2+\cdots\big)f(\lambda) \\
\ee
from which we obtain
\be
\label{DDleg2}
\begin{aligned}
D_1(\lambda,\th,q-1) & \=   1 + \th^2 (q-1)
+\tfrac{1}{12}(6\th^4+2\th^3-5\th^2)(q-1)^2+\cdots
\\
D^\rem_1(\lambda,\th,q-1) & \=
1 - \tfrac{16 \lambda \th + 4\lambda}{\Delta(\lambda)} (q-1) +
\tfrac{2 (-13 \lambda - 1056 \lambda^2 + 1792 \lambda^3) \th 
-3 \lambda - 568 \lambda^2 + 640 \lambda^3}{3\Delta(\lambda)^3} (q-1)^2 + \cdots
\end{aligned}
\ee
with $\Delta(\lambda)=16\lambda-1$. 
Likewise, the expansion of $f(\lambda,q)$ at $q=\z_3$ given by
\be
\Big((1+\z_3\lambda^{\frac{1}{3}})+\big(2\lambda^{\frac{1}{3}}
+8\lambda^{\frac{1}{3}}\th+9(-1+\lambda^{\frac{1}{3}})\th^2
+(2\lambda^{\frac{1}{3}}+4\lambda^{\frac{1}{3}}\th-9\th^2)\z_3\big)(q-\z_3)
+\cdots\Big)f(\lambda)\,.
\ee
More generally, we have expansions
$f(\lambda^{1/m},q) \= D_{m}(\lambda^{1/m},\th,q-\z_m) f(\lambda)$. 
The corresponding Habiro cohomology class $\om_q\in\calH^{1}_\nv(E/B)$
satisfies the $q$-difference equation~\eqref{eq:qPF.v1}.
Notice that this class is different from that of Theorem~\ref{thm.1}
in two ways; first the use of the use of $(q;q)_{2k}/(q;q)_k$ as opposed to
$(q^{1/2};q)_k$, and second the absence of the $B$-normalisation of
Equation~\eqref{omdef}, which in this case does not drastically change
the resulting Habiro cohomology class (see Remark~\ref{rem.factorials}).
The constant term of $\om_q$ at roots of unity is given by
\be
\label{ctleg1}
\om_{m,0} \=
\Big( \sum_{k=0}^{m-1}
\frac{(\z_m;\z_m)_{2k}^2}{(\z_m;\z_m)_{k}^4}\lambda^{\frac{k}{m}} \Big) \om
\,.  
\ee

\subsection{Another version of the Legendre family}
\label{sub.leg2a}

Our next example is a second version of the Legendre family, which like
the one before leads to elements in Habiro cohomology, but unlike before,
the $q$-Picard-Fuchs equation is second order, and not 6th order. Thus,
a fixed Habiro cohomology may have several $q$-holonomic submodules. 

Consider the same family
\be
  y^2\=x(x-1)(x-\lambda)
\ee
where $\lambda \neq 0, 1$.   
The holomorphic one-form is given by $\om=dx/y$.
This form satisfies the following differential
equation
\be
P(\lambda,\th)\om \=0 
\ee  
where $\th=\lambda \pt_\lambda$ as before and 
\be
  P(\lambda,\th)
  \=
  4(\lambda - 1)\th^2 + 4\lambda\th + \lambda
  \=
  \lambda(2\th+1)^2-4\th^2\,.
\ee
Indeed,
\be
  P(\lambda,\th)\om
  \=
  d(-2x^2(x-1)^2\lambda y^{-3})
  \=0\in H^{1}(E/B)\,.
\ee
We can solve the PF equation $P(\lambda,\th)f(\lambda)$ for a power series in
$\lambda$ and find the unique solution 
\be
\label{PFlla}
  P(\lambda,\th)f(\lambda)\=0, \qquad  f(\lambda)
  \=
  \sum_{k=0}^{\infty}
  \frac{(1/2)_k^2}{k!^2}\lambda^k\in\BZ[\![\lambda]\!], \quad f(0)\=1 \,. 
\ee
The $q$-deformation:
\be
\label{leg12}
  f(\lambda,q)
  \=
  \sum_{k=0}^{\infty}
  \frac{(q^{1/2};q)_{k}^2}{(q;q)_{k}^2}\lambda^k\,.
\ee
is the $q$-series considered by Shirai~\cite{Shirai}.
It satisfies the $q$-PF equation
\be
(1-\lambda)f(\lambda,q)-2(1-q^{1/2}\lambda)f(q\lambda,q)
+(1-q\lambda)f(q^2\lambda,q)\=0\,.
\ee
We can use Equation~\eqref{eq:qp.q1} to compute the operator
\be
D_1(\lambda,\th,q-1) \=
1 - \tfrac{1}{2}\th(q-1) + (\tfrac{1}{12}\th^2 + \tfrac{11}{48}\th)(q-1)^2
+ (-\tfrac{7}{96}\th^2 - \tfrac{7}{48}\th)(q-1)^3+\cdots
\ee
that satisfies
\be
f(\lambda,q)\=D_1(\lambda,\th,q-1)f(\lambda)\,.
\ee
Theorem~\ref{thm.1}, adapted to the series~\eqref{leg12} implies
that $\om_q \in \calH^1_\nv(E/B)$ (see Remark~\ref{rem.factorials} for the
slight change in normalisation), and that the
constant term of $\om_q$ when $q=\z_m$ is given by
\be
\label{ctleg2}
\om_{m,0} \=
\Big( \sum_{k=0}^{m-1}
\frac{(\z_m^{1/2};\z_m)_{k}^2}{(\z_m;\z_m)_{k}^2}\lambda^{\frac{k}{m}} \Big) \om
\,.  
\ee
Here we remark that the appearance of $q$ to a fractional power may cause
difficulties in general.

Thus, the two $q$-deformations of the Legendre family in this and the previous
subsections lead to different classes in Habiro cohomology, whose constant terms
given by~\eqref{ctleg1} and~\eqref{ctleg2}, respectively, agree when $m=1$. On
the other hand, these classes generate $q$-holonomic submodules in Habiro cohomology
of ranks 6 and 2, respectively.

\subsection{The quintic 3-fold}
\label{sub.quintica}

The last example illustrating Theorem~\ref{thm.1} will be the Dwork family of
the quintic 3-fold, from which mirror symmetry arose and forever connected
it with Calabi--Yau manifolds~\cite{Candelas:pair}.
A basic reference is the book of Cox-Katz~\cite{Cox-Katz}.
The Picard-Fuchs equation is


\be
\label{PFquintic}
(\th^4 -5 \lambda (5\th + 1)(5\th + 2)(5\th + 3)(5\th + 4))f(\lambda) \=0
\ee
with unique solution 
\be
f(\lambda) \= \sum_{k=0}^\infty \frac{(5k)!}{k!^5} \lambda^k \in \BZ[\![\lambda]\!],
\qquad f(0)=1 \,.
\ee
Its $q$-deformation
\be
f(\lambda,q) \= \sum_{k=0}^\infty \frac{(q;q)_{5k}}{(q;q)_k^5} \lambda^k
\in \BZ[q][\![\lambda]\!]
\ee
satisfies the linear 24th order $q$-difference equation
\begin{tiny}
\be
\label{qPFquintic}
\begin{aligned}
  & (1 - \lambda) f(\lambda)
  + (-4 - q \lambda) f(q \lambda)
  + (6 - q^2 \lambda) f(
  q^2 \lambda) + (-4 - q^3 \lambda) f(q^3 \lambda)
  + (1 - q^4 \lambda) f(q^4 \lambda)
  + q (1 + q) (1 + q^2) \lambda f(q^5 \lambda)
\\ &   
  + q^2 (1 + q) (1 + q^2) \lambda f(q^6 \lambda)
  + q^3 (1 + q) (1 + q^2) \lambda f(q^7 \lambda)
  + q^4 (1 + q) (1 + q^2) \lambda f(q^8 \lambda)
  + q^5 (1 + q) (1 + q^2) \lambda f(q^9 \lambda)
\\ &  
  - q^3 (1 + q^2) (1 + q + q^2) \lambda f(q^{10} \lambda)
  - q^4 (1 + q^2) (1 + q + q^2) \lambda f(q^{11} \lambda)
  - q^5 (1 + q^2) (1 + q + q^2) \lambda f(q^{12} \lambda)
\\ &  
  - q^6 (1 + q^2) (1 + q + q^2) \lambda f(q^{13} \lambda)
  - q^7 (1 + q^2) (1 + q + q^2) \lambda f(q^{14} \lambda)
  + q^6 (1 + q) (1 + q^2) \lambda f(q^{15} \lambda)
\\ &  
  + q^7 (1 + q) (1 + q^2) \lambda f(q^{16} \lambda)
  + q^8 (1 + q) (1 + q^2) \lambda f(q^{17} \lambda)
  + q^9 (1 + q) (1 + q^2) \lambda f(q^{18} \lambda)
  + q^{10} (1 + q) (1 + q^2) \lambda f(q^{19} \lambda)
\\ &  
  - q^{10} \lambda f(q^{20} \lambda)
  - q^{11} \lambda f(q^{21} \lambda)
  - q^{12} \lambda f(q^{22} \lambda)
  - q^{13} \lambda f(q^{23} \lambda)
  - q^{14} \lambda f(q^{24} \lambda) \= 0 \,.
\end{aligned}
\ee
\end{tiny}

We find that
\be
\label{DQ1}
\begin{aligned}
D_{1}(\lambda,\th,q-1) \= & 
1 + 5\th^2 (q-1) + (\tfrac{25}{2}\th^4 + \tfrac{5}{3}\th^3
- \tfrac{25}{12}\th^2)(q-1)^2 \\ & + (\tfrac{125}{6}\th^6 + \tfrac{25}{3}\th^5
- \tfrac{125}{12}\th^4 - \tfrac{5}{3}\th^3 + \tfrac{5}{4}\th^2)(q-1)^3+\cdots
\end{aligned}
\ee
and
\be
\label{DremQ1}
\begin{aligned}
D^\rem_{1}(\lambda,\th,x) \= &
1 + 5\th^2 (q-1) -\frac{1}{12\Delta(\lambda)}
5 (3600 \lambda + 37500 \lambda \th - 5 \th^2 + 
    146875 \lambda \th^2 + 4 \th^3 \\ & + 
    175000 \lambda \th^3)(q-1)^2
    -\frac{1}{12\Delta(\lambda)^3}
  5 (5400 \lambda + 183750000 \lambda^2 \\ & + 76171875000 \lambda^3 + 
    70650 \lambda \th + 1899062500 \lambda^2 \th + 
    699707031250 \lambda^3 \th \\ & + 3 \th^2 + 
    324750 \lambda \th^2 + 6593359375 \lambda^2 \th^2 + 
    1767578125000 \lambda^3 \th^2 - 4 \th^3 \\ & + 
    906250 \lambda \th^3 + 8515625000 \lambda^2 \th^3 + 
    1281738281250 \lambda^3 \th^3)(q-1)^3+\cdots  
\end{aligned}
\ee
with $\Delta(\lambda)=5^5\lambda-1$.


\section{From the Habiro ring of an \'etale map to Habiro cohomology}
\label{sec.method2}

Elements of the Habiro ring of \'etale algebras were constructed in~\cite{GSWZ}
in a variety of methods all stemming from $q$-difference equations of
$q$-hypergeometric sums. One method~\cite[Sec. 4.1]{GSWZ} used a so called
``residue formula'' to construct elements.
We can now rephrase this residue method in terms of Habiro cohomology. We illustrate
the main idea with an example
\be
\label{sumtoy}
f(w,t,q) \=
\sum_{k=0}^{\infty}(-1)^{Ak}\frac{q^{Ak(k+1)/2}w^{Ak}}{(w;q)_{k+1}}t^k
\ee
which is then expanded in around $q=\z_m$. The coefficients of these power series
can be interpreted as rational functions in a similar
sense to the identification of power series in terms of periods in Section
~\ref{sec.method1}. For example, the value at $q=1$ is given by
\be
\sum_{k=0}^{\infty}(-1)^{Ak}\frac{w^{Ak}}{(1-w)^k}t^k \=
\frac{1-w}{1-w-(-1)^Aw^At}\,.
\ee
We can therefore associate to the sum ~\eqref{sumtoy} an element of the Habiro ring
\be
  f(w,t,q)
  \;\in\;
  \calH_{\BZ[t][w^{\pm 1},(1-w-(-1)^Aw^At)^{-1}]/\BZ[t]}\,.
\ee
Taking the residue
\be
\Res f(w,t,q)\frac{dw}{w}\,.
\ee
we then obtain an element of the Habiro ring of the \'etale map
\be
\BZ[t]\to\BZ[t][w^{\pm 1},\delta^{-1}]/(1-w-(-1)^Aw^At)
\ee
We can interpret this element as a class
\be
f(w,t,q)\frac{dw}{w} \in\calH_\nv^{1}(X/B)
\ee
where $X_t$ is determined by the equation $1-w-(-1)^Aw^At$.
In this section, we will generalise this construction by replacing the residue,
which produced algebraic functions, by cohomology classes, which in general will
produce periods.

\subsection{Proof of Theorem~\ref{thm.ct2}}
\label{sub.thm.ct2}

For simplicity, we will assume that $Q=0$ for the proof; however, the more general result follows essentially the same argument.
Fix the \'etale $\BZ[t]$-algebra $R:=\BZ[t][f_\calA(t)]$ with $f_\calA(t)$
as in~\eqref{fcalA}  and recall the proper $q$-hypergeometric series $f_\calA(t,q)$
from Equation~\eqref{fcalAq}. Being proper $q$-hypergeometric, it follows that
it satisfies the system of linear $q$-difference equations
\be
\label{fhsys}
f-\s_\a f = t_\a f - q t_\a \prod_{\a \in \calA} \s_\a f, \qquad \a \in \calA
\ee
with respect to the variables $t_\a$ where $\s_\a(t_\b) = q^{\delta_{\a,\b}} t_\b$.

To prove Theorem~\ref{thm.fA1}, we need to show that: 
\begin{itemize}
\item
  $f_\calA(t,q)$ can be expanded as series in $q-\z_m$ with coefficients in
\newline  $R_m=R[t^{1/m},\z_m]$, 
\item
  these expansions $p$-adically glue after completion and Frobenius twist.
\end{itemize}

To compute the expansion of $f_\calA(t,q)$ near $q=1$, observe first
that by definition and by Equation~\eqref{fcalA}, we have
\be
\label{deltaa}
f_\calA(t,q) \= f_\calA(t) + O(q-1) \,.
\ee

Equation~\eqref{eq:qp.q1} with $q=e^\hbar$ implies that 
\be
\label{qq1}
\begin{aligned}
(q;q)_k \= & (-1)^kk!\hbar^k
  \exp\Big(\sum_{\ell=2}^{\infty}(B_{\ell}(k+1)-B_{\ell}(1))
  \frac{B_{\ell-1}(1)}{\ell}\frac{\hbar^{\ell-1}}{\ell!}\Big) \\
   & \in (-1)^kk!\hbar^k
  \BQ[k][\![q-1]\!]\,.
\end{aligned}
\ee

Equations~\eqref{fcalAq} and \eqref{deltaa} imply that
\be
\label{S11}
f_\calA(t,q) \in S \otimes \BQ[\![q-1]\!]
\ee
where
\be
\label{Sd1}
S\= \BZ[t]\langle f_\calA(t) \rangle
\;:=\; \BZ[t][(\pt_x)^\ell(f_\calA(t)) \, | \ell \geq 0] 
\ee
Since $f_\calA(t)=(1-\sum_{\a \in \calA} t_\a)^{-1}$ is a rational function,
it follows that
\be
\label{SsubR1}
S = R \,.
\ee

On the other hand, Equation~\eqref{SfA1} implies that
\be
\label{AAA1}
f_\calA(t,q) \in \BZ[q^{\pm1}][\![t]\!] \subset
\BZ[\![t]\!][\![q-1]\!]
\ee
Moreover, we claim that
\be
\label{RQ1}
R \otimes \BQ \cap \BZ[\![t]\!] \= R \,.
\ee
Indeed, one inclusion is obvious since
$f_\calA(t) \in \BZ[\![t]\!]$ and the other inclusion is
easy to show (compare with ~\cite[Eqn.(169)]{GSWZ}).

Equations~\eqref{S11}, ~\eqref{SsubR1}, ~\eqref{AAA1} and~\eqref{RQ1} conclude that
$f_\calA(t,q) \in R[\![q-1]\!]$. 

To compute the expansion $f_\calA(t,q)$ near $q=\z_m$,
we use $m$-congruence sums (i.e., taking the summation variable $k\in\BZ_{\geq0}^\calA$
in Equation~\eqref{fcalAq} in a fixed congruence class modulo $m$), we obtain
the following formula for the constant term of the expansion of $f_\calA(t,q)$ at
$q=\z_m$: 
\be
f_\calA(t^{1/m},q) \=
  \frac{1}{1-\sum_\a t_\a}
  \sum_{\ell\in(\BZ_{\geq0}/m\BZ_{\geq0})^\calA}
  \frac{(\z_m;\z_m)_{\sum_\a \ell_a}}{\prod_\a (\z_m;\z_m)_{\ell_\a}}
  \prod_\a t_\a^{\ell_\a/m} + O(q-\z_m) \,.
\ee
Note that the above sum is in $\BZ[t^{1/m},\z_m]$.

We now use the expansion of the $q$-factorial at roots of unity from
Lemma~\ref{lem:qpzm} with $q=\z_me^{\hbar}$ and the equation~\eqref{qpk}
to deduce that
\be
\label{qqm}
\begin{aligned}
  (q;q)_{km+n}
  &\=
  (\z_m;\z_m)_{n}
  (-1)^kk!m^{2k}\hbar^k
  \exp\Big(\sum_{\ell=2}^{\infty}(B_{\ell}(k+1)-B_{\ell}(1))
  \tfrac{B_{\ell-1}(1)(m\hbar)^{\ell-1}}{(-1)^{\ell-1}(\ell-1)\ell!}\\
  &\qquad\qquad\qquad+\sum_{j=0}^{m-1}
  \sum_{\ell=2}^{\infty}
  (B_{\ell}(\tfrac{j+1}{m})-B_{\ell}(\tfrac{j+1}{m}+k+\lceil\tfrac{n-j}{m}\rceil))
  \Li_{2-\ell}(\z_m^j)\tfrac{(m\hbar)^{\ell-1}}{\ell!}\Big)\\
  &\;\in\;
  (-1)^kk!m^{2k}(q-\z_m)^k
  \BQ[\z_m,k][\![q-\z_m]\!]\,.
\end{aligned}
\ee
This and the arguments of the above proof implies that
$f_\calA(t^{1/m},q) \in R_m[\![q-\z_m]\!]$ for all $m \geq 1$.
This proves that $f_\calA(t,q)$ can be expanded for $q$ near a complex root
of unity.

To show that these expansions glue, we use $q$-holonomic methods, namely that
the system of $q$-difference equations~\eqref{fhsys} has a unique solution
as power series in $\BZ[\![t]\!]$ equal to $1+\mathrm{O}(t)$.
By definition $f_\calA(t,q)$, and a fortiori, its expansion around $q=\z_m$
and its re-expansion around $q=\z_{pm}$, and its expansion around $q=\z_{pm}$
satisfy this system of linear $q$-difference equations. 
This implies that the $p$-adic re-expansions also satisfies that same equation with
initial condition $1+\mathrm{O}(t)$ and therefore must be equal after undoing the
$t^{1/m}$ via the Frobenius.

\subsection{Proof of Theorem~\ref{thm.fA1}}
\label{sub.nahm}

We first prove Proposition~\ref{prop.A1}, which 
follows from the well-known $q$-binomial theorem
\be
\label{qbinom}
(x;q)_n \= \sum_{k=0}^n (-1)^k q^{k(k-1)/2}
\left[\!\!\begin{array}{cc} n \\ k \end{array}\!\!\right]_q x^k \,.
\ee
Indeed, we have:
\be
\label{pfA}
\begin{aligned}
\Sf_A(t,q)
  &\=
  \sum_{k,\ell\in\BZ}(-1)^{\diag(A)(k+\ell)}
  \frac{q^{kAk/2-\ell A\ell/2+\diag(A)k/2-\diag(A)\ell/2}}{
    (q;q)_{k}(q^{-1};q^{-1})_{\ell}}t^{k+\ell}\\
  &\=
  \sum_{k,\ell\in\BZ}(-1)^{\diag(A)k+\diag(A-I)\ell}
  \frac{q^{kAk/2-\ell(A-I)\ell/2+\diag(A)k/2-\diag(A-I)\ell/2}}{
    (q;q)_{k}(q;q)_{\ell}}t^{k+\ell}\\
  &\=
  \sum_{m\in\BZ_{\geq0}^N}\sum_{k=0}^{m}
  (-1)^{\diag(I)k+\diag(A-I)m}
  \frac{q^{kIk/2+\diag(1)k/2-m(A-I)m/2+m(A-I)k+\diag(A)k-\diag(A-I)m/2}}{
    (q;q)_{k}(q;q)_{m-k}}t^{m}\\
  &\=
  \sum_{m\in\BZ_{\geq0}^N}
  (-1)^{\diag(A-I)m}q^{-m(A-I)m/2-\diag(A-I)m/2}
\prod_{j=1}^N
\left[\!\!\begin{array}{cc} (Am)_j+A_{jj}-1  \\ m_j \end{array}\!\!\right]_q
t^{m}\,.
\end{aligned}
\ee
\qed

Proposition~\ref{prop.A1} implies the following. Let $\s_j$ denote the operator
that shifts $t_j$ to $qt_j$ and $t_i$ (for $i \neq j$) to $t_i$, for $j=1,\dots,N$. 

\begin{corollary}
\label{cor.SfA}
$\Sf_A(t,q)$ is $q$-holonomic and satisfies the system of linear
$q$-difference equations
\be
\label{eq:qdiff.sym.nahm}
\begin{aligned}
  &((1-\sigma_i)\prod_{j=1}^{N}(q^{A_{jj}-(A-I)_{ij}}
  \sigma^{(A-I)m};q)_{A_{ij}-1}\Sf_A)(t,q)\\
  &\=
  (-1)^{(A-I)_{ii}}q^{-(A-I)_{ii}}t_i\sigma^{(A-I)_{i\cdot}}
  \prod_{j=1}^{N}(q^{A_{jj}}\sigma^{Am};q)_{A_{ij}}\Sf_A)(t,q), \qquad i=1,\dots,N\,.
\end{aligned}
\ee
\end{corollary}

\begin{proof}
The quotient of the summand of the RHS of Equation~\eqref{SfA1} at $m+\delta_i$
by that at $m$ \hbox{equals to}
\be
  (-1)^{(A-I)_{ii}}q^{-((A-I)m)_i-(A-I)_{ii}}\frac{t_i}{1-q^{m_i+1}}
  \prod_{j=1}^N
  \frac{(q^{(Am)_j+A_{jj}};q)_{A_{ij}}}{(q^{(Am)_j+A_{jj}-m_j};q)_{A_{ij}-1}}\,.
\ee
This and Proposition~\ref{prop.A1} implies the corollary.
\end{proof}  

\begin{proof}[Proof of Theorem~\ref{thm.fA1}]
The proof will follow the steps of the proof of Theorem~\ref{thm.ct2} with
some modifications due to the fact that we are dealing with algebraic rather than
with rational functions. To prove Theorem~\ref{thm.fA1}, we need to show that:
\begin{itemize}
\item
  $\Sf_A(t,q)$ can be expanded as series in $q-\z_m$ with coefficients in
  $R_m=R[t^{1/m},\z_m]$,
\item
  these expansions $p$-adically glue after completion and Frobenius twist.
\end{itemize}

To compute the expansion of $\Sf_A(t,q)$ near $q=1$, we need the following
identity~\cite[Eq. 4.1.5-4.1.7]{RV:Apoly}
\be
\label{deltaid}
  \frac{1}{\delta(t)}
  \=
  \sum_{k\in\BZ_{\geq0}^N}
  (-1)^{\diag(A-I)k} \prod_{j=1}^N
  \binom{(Ak)_j+A_{jj}-1}{k_j}
  t_1^{k_1} \dots t_N^{k_N}
\ee
which together with Equation~\eqref{SfA1} implies that
$\Sf_A(t,q)=\frac{1}{\delta(t)} + O(q-1)$.

Equations~\eqref{SfA1}, ~\eqref{deltaid} and~\eqref{qq1} imply that
\be
\label{S1}
\Sf_A(t,q) \in S \otimes \BQ[\![q-1]\!]
\ee
where
\be
\label{Sd}
S\= \BZ[t^{\pm 1}]\langle \delta(t)^{-1} \rangle
\;:=\; \BZ[t^{\pm 1}][(\pt_t)^\ell(\delta(t)^{-1}) \, | \ell \geq 0] 
\ee
We claim that
\be
\label{SsubR}
S \subset R \,.
\ee
Indeed, Equation~\eqref{deltadef} implies that
$\delta \in \BZ[z^{\pm 1}]$ where $z^{\pm 1}=(z_1^{\pm 1},\dots,z_N^{\pm 1})]$.
It follows that
$\pt_t (\delta^{-1}) \in \BZ[\delta^{-1},z^{\pm 1}, (\pt_t z)^{\pm 1}]$ and
$\pt_t z_i \in \BZ[t^{\pm 1},\delta^{-1},z^{\pm 1}]$ from the definition of
$\delta$.
This proves~\eqref{SsubR}.

On the other hand, Equation~\eqref{SfA1} implies that
\be
\label{AAA}
\Sf_A(t,q) \in \BZ[q^{\pm1}][\![t]\!] \subset \BZ[\![t]\!][\![q-1]\!]
\ee
Moreover, we claim that
\be
\label{RQ}
R \otimes \BQ \cap \BZ[\![t]\!] \= R \,.
\ee
Indeed, the series $z_j(t)$ of Equation~\eqref{XA} and $\delta(t)$ from
Equation~\eqref{deltadef} satisfy $z_j(t)=1+O(t)$ for $j=1,\dots,N$ and
$\delta(t)=1 +O(t)$, which implies that $R \subset \BZ[\![t]\!]$ proving one
inclusion in~\eqref{RQ}. The other inclusion is proven as in~\cite[Eqn.(169)]{GSWZ}.

Equations~\eqref{S1}, ~\eqref{Sd}, ~\eqref{AAA} and~\eqref{RQ} conclude that
$\Sf_A(t,q) \in R[\![q-1]\!]$. 

To compute the expansion $\Sf_A(t,q)$ from Equation~\eqref{SfA1} near $q=\z_m$,
we use $m$-congruence sums (i.e., taking the summation variable $k\in\BZ_{\geq0}^N$
in Equation~\eqref{SfA1} in a fixed congruence class modulo $m$) to obtain
the following formula for the constant term of the expansion of $\Sf_A$ at
$q=\z_m$:
\be
\begin{aligned}
  \Sf_A(t^{1/m},q)
  &\=
  \frac{1}{\delta(t)}
  \sum_{\ell\in(\BZ_{\geq0}/m\BZ_{\geq0})^N}
  (-1)^{\diag(A-I)\ell}\z_m^{-\ell(A-I)\ell/2-\diag(A-I)\ell/2}\\
  &\qquad\qquad\qquad\times\prod_{j=1}^N
  \left[\!\!\begin{array}{cc} (A\ell)_j+A_{jj}-1 \\ \ell_j
    \end{array}\!\!\right]_{\z_m} t_1^{\ell_1/m} \dots t_N^{\ell_N/m} + O(q-\z_m) \,.
\end{aligned}
\ee
Note that the above sum is in $\BZ[t^{1/m},\z_m]$.
Then, the expansion of the $q$-binomial at roots of unity~\eqref{qqm}
and the arguments of the above proof implies that
$\Sf_A(t^{1/m},q) \in R_m[\![q-\z_m]\!]$ for all $m \geq 1$. This proves that
$\Sf_A(t,q)$ can be expanded for $q$ near a complex root of unity.

To show that these expansions glue, we use $q$-holonomic methods, namely that
the system of $q$-difference equations~\eqref{eq:qdiff.sym.nahm} has a unique solution
as power series in $t$ equal to $1+\mathrm{O}(t)$. By definition $\Sf_A(t,q)$ satisfies
this system of linear $q$-difference equations.
Therefore, the expansions for $q$ near $\z_m$ also satisfy this difference equation.
This implies that the $p$-adic re-expansions also satisfies that same equation with
initial condition $1+\mathrm{O}(t)$ and therefore must be equal after undoing the
$t^{1/m}$ via the Frobenius.
\end{proof}

\subsection{Push-forward to cohomology}
\label{sub.pf}

In this section, we will prove Theorem~\ref{thm.push}.
The result is almost obvious but we need to check that the series converge.

\begin{proof}[Proof of Theorem~\ref{thm.push}] 
Fix $f(q)\in\calH_{R/\BZ[x,\lambda]}^{\mathrm{an}}$ and $\om\in\Om^{n}(X)$ so
that $\varphi_p (\om)=\om$, where $X=\Spec(R)$ and $B=\Spec(\BZ[\lambda,1/\Delta])$
(where $\lambda$ is a possible empty set of variables).
Then, we see that $\Om^{n}(X/B)$ is an $R$-module and we have
$f(q)\om\in\calH(\Om^{n}(X/B))$.
We can then reduce this to give an element of
\be
f(q)\om\in \prod_m H^{n}_\dR(X/B)[\z_m][\![q-\z_m]\!]\,.
\ee
The assumption that $f\in\calH_{R/\BZ[x,\lambda]}^{\mathrm{an}}$, implies
that the growth of the denominators appearing in the expansion of $f$ as power
series in $q-\z_m$ is linear.
Therefore, we see that, after inverting $\Delta$, applying Griffiths
reduction~\cite{Griffiths} introduces logarithmically growing valuations.
Therefore, the series are convergent and therefore their re-expansions must agree.
This concludes the proof of the theorem. 
\end{proof}

In the rest of this section, we revisit the examples from subsections
~\ref{sub.leg1a}, ~\ref{sub.leg1b} and ~\ref{sub.quintica} and apply the second
method to obtain elements of the Habiro cohomology. Our examples illustrate both 
Theorem~\ref{thm.push}, and the fact that the two methods of Theorems~\ref{thm.1}
and~\ref{thm.push} constructing elements in the Habiro cohomology are complementary,
and lead to the same elements. 

\subsection{The Legendre family}
\label{sub.leg1b}

Proposition~\ref{prop.A1} with $A=2$ and a slight shift $t \mapsto q^{-1}t$ imply
that
\be
\Sf(t,q)
\=
\Big(\sum_{m=0}^{\infty} \frac{q^{m^2}}{(q;q)_m}t^{m}\Big)
\Big(\sum_{\ell}^{\infty}
\frac{q^{-\ell(\ell+1)}}{(q^{-1};q^{-1})_{\ell}}t^{\ell}\Big)
  \=
  \sum_{k=0}^{\infty}(-1)^{k}q^{-k(k+1)/2}
  \left[\!\!\begin{array}{cc} 2k \\ k \end{array}\!\!\right]_qt^k \,.
\ee
This and Theorem~\ref{thm.fA1} defines an element $\Sf(t,q)$ of the Habiro ring of
the \'etale map
\be
\BZ[t] \to \BZ[t][z^{\pm1}]/(z^2-(1+4t))\,.
\ee
It follows that $\Sf(x,q)\Sf(x^{-1}\lambda,q)$ is an element of the
Habiro ring of the \'etale map 
\be
\label{xlaz}
  \BZ[\lambda,x^{\pm1}]
  \to
  \BZ[\lambda,x^{\pm1}][z^{\pm1}]/(z^2-(1+4x)(1+4\lambda x^{-1}))\,.
\ee
We can then define a class in the Habiro cohomology 
\be
\om_q = \Sf(\lambda,q)\Sf(x^{-1}\lambda,q)\frac{dx}{x} \in \calH^1_\nv(X/B)
\ee
where $X/B$ is the family of elliptic curves defined by
$z^2 = (1+4x)(1-4\lambda x^{-1})$.
Since
\be
\label{Sfxx}
\Sf(x,q)\Sf(x^{-1}\lambda,q) \=
\sum_{k,\ell=0}^{\infty}(-1)^{k+\ell}q^{-k(k+1)/2-\ell(\ell+1)/2}
\left[\!\!\begin{array}{cc} 2k \\ k \end{array}\!\!\right]_q
\left[\!\!\begin{array}{cc} 2\ell \\ \ell \end{array}\!\!\right]_q
x^{k-\ell}\lambda^\ell 
\ee
it follows that $\om_q$ satisfies the same $q$-difference equation in $\lambda$
as the $x$-constant term in the series~\eqref{Sfxx}, namely the series
\be
  \sum_{k=0}^{\infty}
  q^{-k(k+1)}
  \left[\!\!\begin{array}{cc} 2k \\ k \end{array}\!\!\right]_q^2
  \lambda^k\,.
\ee
This gives a different $q$-deformation to that of section~\ref{sub.leg1a}.
The class in section~\ref{sub.leg1a} can be constructed in a similar way as a
push-forward by taking
\be
\label{f111}
  f(\lambda,x,q)
  \=
  \sum_{k,\ell=0}^{\infty}
  \left[\!\!\begin{array}{cc} 2k \\ k \end{array}\!\!\right]_q
  \left[\!\!\begin{array}{cc} 2\ell \\ \ell \end{array}\!\!\right]_q
  \lambda^kx^{\ell-k}\,.
\ee
Remark~\ref{rem.twist} implies that this gives an element of the Habiro ring of
the \'etale map
\be
\BZ[\lambda,x^{\pm1}]\to
\BZ[\lambda,x^{\pm 1}][y^{\pm1}]/(y^2-(1-4x)(1-4\lambda x^{-1}))\,.
\ee
The specialization $f(\lambda,x)=f(\lambda,x,1)$ is the algebraic function
\be
f(\lambda,x) \= \sum_{k,\ell=0}^{\infty} \binom{2k}{k}^2 \binom{2\ell}{\ell}^2 
\lambda^kx^{\ell-k} = \frac{1}{\sqrt{(1-4x)(1-4\lambda x^{-1})}} = \frac{1}{y}
\ee
on the family of elliptic curves $y^2-(1-4x)(1-4\lambda x^{-1})$.

On the other hand, computing the first few terms, we find
\be
\begin{aligned}
  f(\lambda,x,q)
  &\=
  (1+\tfrac{1}{2}(2\th_\lambda^2 + 2\th_x\th_\lambda + \th_x^2)(q-1)\\
  &\qquad+\tfrac{1}{24}\big(12\th_\lambda^4 + (24\th_x + 4)\th_\lambda^3
  + (24\th_x^2 + 6\th_x - 10)\th_\lambda^2\\
  &\qquad + (12\th_x^3 + 6\th_x^2 - 10\th_x)\th_\lambda
  + (3\th_x^4 + 2\th_x^3 - 5\th_x^2)\big)(q-1)^2+\cdots)f(\lambda,x)\,.
\end{aligned}
\ee
which after reduction, implies that
\be
\begin{aligned}
  f(\lambda,x,q)
  &\=
  \frac{1}{y}+
  \frac{1}{y^5}\big(x - 16\lambda + \lambda x^{-1}(1+2\lambda x^{-1})
  + x^2 (2 + 64\lambda^2 x^{-3})\big)(q-1)\\
  &\qquad+\frac{1}{y^9}\Big((11 x^2 + 12 x^3 - 14 x^4)
  + (x - 182 x^2 - 192 x^3 + 256 x^4) (\lambda x^{-1})\\
  &\qquad+ (11 - 182 x + 2148 x^2 - 1664 x^3 + 1280 x^4) (\lambda x^{-1})^2\\
  &\qquad+ (12 - 192 x - 1664 x^2 - 6144 x^3 + 6656 x^4) (\lambda x^{-1})^3\\
  &\qquad+ (-14 + 256 x + 1280 x^2 + 6656 x^3 - 6144 x^4) (\lambda x^{-1})^4\Big)
  (q-1)^2+\cdots\,.
\end{aligned}
\ee
Pushing forward, and illustrating Theorem~\ref{thm.push}, we obtain the
element
\be
  f(\lambda,x,q)\frac{dx}{x} \in \calH^1_\nv(X/B)\,,
\ee
which agrees with the element in the Habiro cohomology constructed in
Section~\ref{sub.leg1a}.

\subsection{The quintic 3-fold}
\label{sub.quinticb}

The last example in this section illustrating Theorem~\ref{thm.push} is the
quintic 3-fold. Let
\be
f_\psi(x) \= \psi x_1x_2x_3x_4x_5 -(x_1^5+x_2^5+x_3^5+x_4^5+x_5^5)
\ee
for $x=(x_1,\dots,x_5)$. Consider the (homogeneous, degree $0$) rational function
$g_\psi(x)$ and differential form $\Omega_0$ given by
\be
g_\psi(x) \= \frac{\psi x_1 x_2 x_3 x_4 x_5}{f_\psi(x)},
\qquad
\Omega_0 \= \sum_{j=1}^5 (-1)^{j-1} \frac{dx_1}{x_1}\wedge \dots
\wedge \widehat{\frac{dx_j}{x_j}} \wedge \dots \wedge \frac{dx_5}{x_5}
\ee
\be
g_\psi(x) =\!\!
  \sum_{m_i\in\BZ_{\geq0}^{5}}
  \binom{m_1+\cdots+m_5}{m_1,\dots,m_5}x_1^{4m_1-m_2-\cdots-m_5}
  \cdots x_5^{4m_5-m_1-\cdots-m_4}\psi^{-m_1-\cdots-m_5}
\ee
using the standard multinomial notation.
Since $g_\psi(x) \in \BQ(\psi,x)$, by Remark \ref{rem.twist} it follows that 
its $q$-deformation
\be
\begin{aligned}
g_\psi(x,q) &=\!\!
  \sum_{m_i\in\BZ_{\geq0}^{5}}
  \left[\!\!\begin{array}{cc} m_1+\cdots+m_5 \\ m_1,\cdots,m_5
    \end{array}\!\!\right]_q\!\!
  x_1^{4m_1-m_2-\cdots-m_5}\cdots x_5^{4m_5-m_1-\cdots-m_4}\psi^{-m_1-\cdots-m_5}.
\end{aligned}
\ee
defines an element in the Habiro ring
$\calH_{\BZ[\psi^{\pm 1},x][f_\psi(x)^{-1}]/\BZ[\psi^{\pm 1},x]}$.
Let $X/B$ denote the family of affine 3-folds defined by $f_\psi(x)=0$ over
$B=\Spec(\BZ[\psi^{-5},\tfrac{1}{\Delta(\psi^{-5})}])$ where
$\Delta(\lambda)=5^5 \lambda -1$.
We can then define a class in $\calH_\nv^3(X/B)$ by
\be
\begin{aligned}
  &\om_q
  \=
  \Res \, \big( g_\psi(x,q) \Omega_0 \big)\,.
\end{aligned}
\ee
Here, the residue map along $f_\psi(x)$ is the one discussed in Cox--Katz
~\cite[Sec.5.3]{Cox-Katz}. 

Now $\Res g_\psi(x) \Omega_0$ satisfies the same linear differential
equation as $\Res_{x=0}g_\psi(x)$, where $\Res_{x=0}$ means taking the
constant term (see Equation~\eqref{roughly}) and
\be
\Res_{x=0}g_\psi(x) \= \sum_{m\in\BZ_{\geq0}} \frac{(5m)!}{m!^5} \psi^{-5m} \,.
\ee
A detailed proof of this fact is given in~\cite[Example 5.4.1]{Cox-Katz}.

It follows that $\om_q$ satisfies the same $q$-difference equation as the sum
\be
  \sum_{m\in\BZ_{\geq0}} \frac{(q;q)_{5m}}{(q;q)_m^5}
  \psi^{-5m}\,.
\ee
Up to a change of variables $\psi^{-5} \mapsto \lambda$, this push-forward gives
rise to the same class in Habiro cohomology as the one in subsection
~\ref{sub.quintica}. 


\section{The 3D-index of the $4_1$ knot}
\label{sec.41}

In this section we discuss in detail the 3D-index of the
simplest hyperbolic $4_1$ knot (a sequence of $q$-hypergeometric series) whose
experiments presented briefly in the introduction provably produce elements
of the first Habiro cohomology of an elliptic curve $X$ (the $A$-polynomial of the
$4_1$ knot). Our presentation of the 3D-index of the $4_1$ knot will
follow~\cite{GW:periods} and will blend experiments, statements and theorems.

To simplify the notation, we use $(x,y)$ instead of $(z_1,z_2)$
in Equation~\eqref{RA} and likewise $(x,y,z)$ instead of $(z_1,z_2,z_3)$.

\subsection{Unsymmetrised series}

Consider the collection of $q$-hypergeometric series
\be
\label{hndef}
h_{n}(q) \= \sum_{k\in\BZ}
(-1)^{k+n}\frac{q^{k(k+1)/2+nk+n(2n+1)/2}}{(q;q)_{k}(q;q)_{k+2n}},
\qquad (n \in \BZ) \,.
\ee
from~\cite[Prop.8]{GW:periods}. This is one of the two colored descendant
holomorphic blocks of the $4_1$ knot.
Using the variables $(q^k,q^n)=(x,y)$, we find that the 
Nahm equations of $h_n(q)$ (that come from the summand of $h_n(q)$) are given by
\be
\label{nahmdesc}
xy+(1-x^2y)(1-y)\=0\,.
\ee
This equation defines an affine curve $X=X_{xy}=\Spec(R)$\footnote{the ring $R$ in
  this section should not be confused with any rings with the same name from previous
  sections.}
in $\mathbb{G}_m^2$
with coordinates $(x,y)$ and gives rise to an \'etale map
\be
\label{41et}
\BZ[x^{\pm 1}]\to R:=
\BZ[x^{\pm 1}][y^{\pm 1},\tfrac{1}{\delta}]/(xy+(1-x^2y)(1-y))
\ee
and to a smooth map $X/\mathbb{G}_m$, 
where 
\be
\label{delta2}
\delta^2(x) \=x^{2} - 2x - 1 - 2x^{-1} + x^{-2}\,.
\ee
It is easy to see that $X/\Spec(\BZ[1/30])$ is a smooth affine elliptic curve.

A numerical computation of the asymptotics of $h_n(q)$ with
$q=e^\hbar$ as $\hbar \to 0$ are given by
\be
h_{n}(e^\hbar) \sim
f_{X/\mathbb{G}_m}(e^\hbar)|_{x=e^{n \hbar}}\,,
\ee
where
\be
\label{fXvals}
\begin{aligned}
  f_{X/\mathbb{G}_m}(e^\hbar)
  \= e^{\tfrac{V(x)}{\hbar}}
  \frac{1}{\sqrt{\delta}}\Big(1&+
  \frac{\hbar}{24\delta^3}(x^3 - x^2 - 2x + 15 - 2x^{-1} - x^{-2} + x^{-3})\\
  &+\frac{\hbar^2}{1152\delta^{6}}(x^6 - 2x^5 - 3x^4 + 610x^3 - 606x^2 - 1210x + 3117\\
  &- 1210x^{-1} - 606x^{-2} + 610x^{-3} - 3x^{-4} - 2x^{-5} + x^{-6})+\cdots\Big)\,.
\end{aligned}
\ee
We have computed 150 coefficients of the above series.
%
%
Note that the specialization to $x=1$ of the above series is the perturbative
Chern--Simons series of the $4_1$ knot few terms of which were given
in~\cite[Eqn.1.3]{GZ:kashaev} and~\cite[Eqn.1]{GSWZ}.

Aside from numerical computations, formal Gaussian integration from~\cite{DG},
applied to the $q$-hypergeometric sum~\eqref{hndef}, works and defines a collection
$f_{X/\mathbb{G}_m}(q)$ of power series around each root of unity. 

\begin{remark}
\label{rem.1}
The arguments of~\cite{GSWZ} (together with an appropriate extension of relative
$K$-theory of the \'etale map~\eqref{41et}) ought to imply that
$f_{X/\mathbb{G}_m}(q)$ is a section of a line bundle over the Habiro ring
$\calH_{R/\BZ[x^{\pm1}]}$. 
\end{remark}

\subsection{Symmetrised series}

We consider next the symmetrisation
\be
\label{fsdef}
\Sf_{X/\mathbb{G}_m}(q) = f_{X/\mathbb{G}_m}(q) f_{X/\mathbb{G}_m}(q^{-1})
\ee
of $f_{X/\mathbb{G}_m}(q)$. This is motivated on the one hand by the 3D-index of
$4_1$ (which is a bilinear expression on colored holomorphic blocks,
see e.g.,~\cite[Prop.9]{GW:periods}), and on the other hand by~\cite{GSWZ}
as well as by Section~\ref{sub.exphab}.

\begin{theorem}
\label{thm.2}  
$\Sf_{X/\mathbb{G}_m}(q)$ is an element of the Habiro ring
$\calH_{R/\BZ[x^{\pm1}]}$.
\end{theorem}

For example, the expansion of the above element near $q=1$ is given by
\be
\label{fsymval}
\Sf_{X/\mathbb{G}_m}(q) \=
\frac{1}{\delta}
\big(1+
\frac{(q-1)^2}{\delta^{6}}(x^3 - x^2 - 2x + 5 - 2x^{-1} - x^{-2} + x^{-3})
+ \dots \big) \,.
\ee

\begin{proof}
As was discussed in~\cite[Sec.4.1]{GSWZ}, we can give a residue formula
for the symmetrisation $\Sf_{X/\mathbb{G}_m}(q)$ that avoids formal Gaussian
integration, proves integrality and gluing and gives a self-contained proof of
Theorem~\ref{thm.2}. This residue formula is
\be
\label{fsres}
\Sf_{X/\mathbb{G}_m}(q) \= \sum_{z^m=y}
\underset{w=z}{\Res}
\sum_{k=0}^{\infty}
(-1)^{k}\frac{q^{k(k+1)/2}w^kx^k}{(qw;q)_{k}(qwx^2;q)_{k}}
\frac{dw}{w} \in R[x^{1/m},\z_m][\![q-\z_m]\!]
\ee
where $y$ satisfies~\eqref{nahmdesc}. The arguments
in~\cite[Sec.4.1]{GSWZ} show that the collection of series ~\eqref{fsres} is
well-defined and $p$-adically glues, thus defines an element of the Habiro ring  
$\calH_{R/\BZ[x^{\pm1}]}$.

An alternative proof that the series~\eqref{fsres} are in the Habiro ring
$\calH_{R/\BZ[x^{\pm1}]}$ can be obtained as an application of Theorems~\ref{thm.ct2}
and~\ref{thm.push} as follows.
The proof of Theorem~\ref{thm.ct2} implies that
\be
\label{gsres}
\sum_{k=0}^{\infty}
(-1)^{k}\frac{q^{k(k+1)/2}w^kx^k}{(qw;q)_{k}(qwx^2;q)_{k}}
\in S[w^{1/m}, x^{1/m},\z_m][\![q-\z_m]\!]
\ee
lies in the Habiro ring $\calH_{S/\BZ[w^{\pm 1},x^{\pm 1}]}$ where
$S:=\BZ[w^{\pm 1},x^{\pm 1}][(1+ \tfrac{wx}{(1-w)(1-wx^2)})^{-1}]$. Using
$\om=\frac{dw}{w}$ in Theorem~\ref{thm.push}, it implies that~\eqref{fsres}
lies in $\calH_{R/\BZ[x^{\pm1}]}$.
\end{proof}

\subsection{Summed symmetrised series}

So far we discussed one of the descendant holomorphic blocks~\eqref{hndef} of
the $4_1$ knot. The (meromorphic) 3D-index however, is a sum over the integers of a
bilinear combination such blocks; see~\cite[Eqn.(34)]{GW:periods}. In other words,
the $3D$-index of the $4_1$ knot is related to the series
\be
\label{3dsum}
  \sum_{n\in\BZ}
  h_{n}(q)h_{n}(q^{-1})
  \=
  \sum_{n,k,\ell\in\BZ}
  (-1)^{k+\ell}\frac{q^{k(k+1)/2+\ell(\ell+1)/2+nk+n\ell+n(2n+1)}}{
    (q;q)_{k}(q;q)_{k+2n}(q;q)_{\ell}(q;q)_{\ell+2n}}\,.
\ee
Geometrically, summation over $n$ corresponds to integration over $x=q^n$ and
this results to a push-forward to an affine scheme $Y$, defined shortly, which
is a copy of the elliptic curve $X/\mathrm{Spec}(\BZ[1/30])$ plus 4 additional
points. Indeed, with $(q^n,q^k,q^\ell)=(x,y,z)$, the Nahm equations associated
to the $q$-series ~\eqref{3dsum} are given by
\be
\label{xyznahm}
  \frac{x^4yz}{(1-x^2y)^2(1-x^2z)^2}\=1\,, \quad 
  -\frac{xy}{(1-x^2y)(1-y)}\=1\,,\quad
  -\frac{xz}{(1-x^2z)(1-z)}\=1\,.
\ee
Geometrically, the above scheme is given by
\be
Y=X_{xy} \times_{\mathbb{G}_{m,x}} X_{xz} 
\ee
where $X_{xy}/\mathbb{G}_{m,x}$ and $X_{xz}/\mathbb{G}_{m,x}$ is the affine
curve~\eqref{nahmdesc} in coordinates $(x,y)$ and $(x,z)$, respectively.

It turns out that
\be
Y/\mathbb{G}_{m,x} = X \sqcup \{p_1,p_2,p_3,p_4\}
\ee
where
\be
\label{4p}
p_1=(1,\z_6,\z_6), \quad p_2=(1,\z_6^{-1},\z_6^{-1}), \quad
p_3=(-1, \tfrac{3 + \sqrt{5}}{2}, \tfrac{3 + \sqrt{5}}{2}), \quad
p_4=(-1, \tfrac{3 - \sqrt{5}}{2}, \tfrac{3 - \sqrt{5}}{2})
\ee
consists of a copy of the elliptic curve $X$ plus 4 isolated points
with $x=\pm1$. Indeed, away from the 4 isolated points, Equations~\eqref{xyznahm}
are equivalent to
\be
-\frac{xy}{(1-x^2y)(1-y)}\=1\,,\qquad z\=(1 - y - x y + y^2 + x y^2)/(y-1)^2 \,.
\ee
The projection $(x,y,z) \mapsto x$ of $Y$ to $\mathbb{G}_{m,x}$ is shown
in Figure ~\ref{fig.Y} where the 4 branch points $b_j$ for $j=1,\dots, 4$
of $Y$ are given by
\be
\label{4b}
b_1=(\z_3,\z_6,\z_6), \quad b_2=(\z_3^{-1},\z_6^{-1},\z_6^{-1}), \quad
b_3=(\tfrac{3 - \sqrt{5}}{2}, \tfrac{3 + \sqrt{5}}{2}, \tfrac{3 + \sqrt{5}}{2}), \quad
b_4=(\tfrac{3 + \sqrt{5}}{2}, \tfrac{3 - \sqrt{5}}{2}, \tfrac{3 - \sqrt{5}}{2}) \,.
\ee

\begin{figure}[htpb!]
\begin{center}
\begin{tikzpicture}[scale=0.6]
\draw(-6,0)--(6,0) node[right]{$x\in\mathbb{G}_m$};
\filldraw (-3.5,0) circle(2pt) node[below] {$x=1$};
\filldraw (-3.5,3.5) circle(2pt) node[left] {$p_1$};
\filldraw (-3.5,2.5) circle(2pt) node[left] {$p_2$};
\filldraw (3.5,0) circle(2pt) node[below] {$x=-1$};
\filldraw (3.5,3.5) circle(2pt) node[right] {$p_3$};
\filldraw (3.5,2.5) circle(2pt) node[right] {$p_4$};
\filldraw (-4.858,3) circle(2pt) node[left] {$b_1$};
\filldraw (-2,3) circle(2pt) node[right] {$\;b_2$};
\filldraw (2,3) circle(2pt) node[left] {$b_3\;$};
\filldraw (4.858,3) circle(2pt) node[right] {$b_4$};
\draw[ultra thick] (-6,4) to[out angle=0, in angle=180,
    curve through = {(-4,2) (-2,3) (0,4) (2,3) (4,2)}] (6,4);
\draw[ultra thick] (-6,2) to[out angle=0, in angle=180,
    curve through = {(-4,4) (-2,3) (0,2) (2,3) (4,4)}] (6,2);
\draw(6,2) node[right] {$X$};
\end{tikzpicture}
\end{center}
\caption{The projection $Y\to\mathbb{G}_m$.}
\label{fig.Y}
\end{figure}
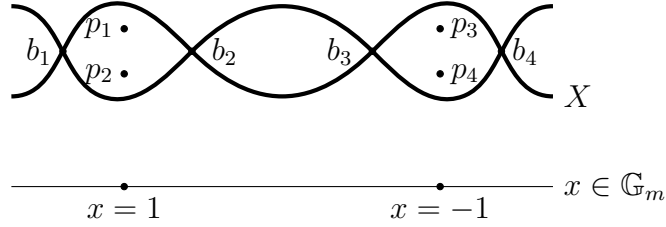


But now, formal Gaussian integration cannot be performed on the sum~\eqref{3dsum}
because the Hessian vanishes identically on $X\subset Y$. (Formal Gaussian
integration can be performed at the
4 isolated points $p_j$ and $p_1$ and $p_2$ contribute to the asymptotics, whereas
the other two do not; see~\cite[Eqn.(3)]{GW:periods}).

Let us explain this problem briefly. To perform
formal Gaussian integration, one replaces the quantum factorials from the denominator
of the $q$-series~\eqref{3dsum} by infinite Pochhammer symbols in the
numerator, and then converts the sum over a 3-dimensional lattice and apply Poisson
summation. Ignoring all exponentially small terms, there is a single term remaining
(from the origin of the 3-dimensional lattice), and this leads to a single
formal Gaussian integral. However, the problem is that the critical points of
the integrand have a positive dimensional component.

In this situation one would then expect the series associated to that component to
be periods coming from the middle dimensional homology of this component.
This is exactly what we found numerically in~\cite{GW:periods}. 
The vertical asymptotics of the 3d-index when $q=e^{2\pi i \tau}=e^{\hbar}$ and
$\tau \in i \BR_+ \to 0$ results in the series $\Psi(\hbar)$ discussed in
the introduction (see Equation~\eqref{Psival}), where
\be
\label{41asy}
\Psi(\hbar) \=
\frac{1}{\hbar}\int_{\mathcal{C}} g_{X/\BZ[1/30]}(e^\hbar) \= 
\frac{1}{\hbar}\int_{\mathcal{C}}
\big(\omega + \omega' \hbar^2 +O(\hbar^3) \big)
\= \frac{1}{\hbar} \kappa + \hbar \kappa' +O(\hbar^2)
\ee
where
the integrating form is given by
\be
\label{gxdef}  
g_{X/\BZ[1/30]}(q) \= f_{X/\mathbb{G}_m}(q)f_{X/\mathbb{G}_m}(q^{-1})\frac{dx}{x} \,,
\ee
the forms $\om$ and $\om'$ are given in~\eqref{om2}, 
the contour $\calC$ is given in Figure~\ref{fig.calC} in logarithmic coordinates
$u=\log(x)$, has branch points at
\be
  u\=\pm0.96242\cdots\,\quad u\=\pm 2.0944\cdots i\,,
\ee
and is decomposed into Henkel contours, and 
the periods are given numerically by
\be
\begin{aligned}
  \kappa
  & \= \int_{\mathcal{C}}\omega
  \= -\int_{\mathcal{C}_2}\omega
  \=1.400603042\cdots \,,
  & \int_{\mathcal{C}_1}\omega
  & \=-\kappa - 1.596242222\cdots i\,,
  \\
  \kappa'
  & \=  \int_{\mathcal{C}}\omega'
  \=-\int_{\mathcal{C}_2}\omega'
  \= 0.008005298\cdots \,,
  &
  \int_{\mathcal{C}_1}\omega'
  &\=-\kappa'+ 0.011257603\cdots i\,,
\end{aligned}
\ee
matching with the numerical asymptotics of \cite[Equs. (4),(6)]{GW:periods}
where $\pi^{-1}\kappa=0.4458257950\cdots$ and $4\pi\kappa'=0.1005975491\cdots$.

\begin{figure}[htpb!]
\begin{center}
\begin{tikzpicture}
\draw(-6,0)--(6,0);
\draw(0,-3)--(0,3);
\filldraw (-0.96242,0) circle(2pt);
\filldraw (0.96242,0) circle(2pt);
\filldraw (0,-2.0944) circle(2pt);
\filldraw (0,2.0944) circle(2pt);
\draw[red,snake it] (-0.96242,0)--(0.96242,0);
\draw[red,snake it] (0,-2.0944)--(0,2.0944);
\draw[red,snake it] (0,1.0472) to[out angle=0, in angle=180,
    curve through = {(3,1.4)}] (6,1.5);
\draw[red,snake it] (0,1.0472) to[out angle=180, in angle=0,
    curve through = {(-3,1.4)}] (-6,1.5);
\draw[red,snake it] (0,-1.0472) to[out angle=0, in angle=180,
    curve through = {(3,-1.4)}] (6,-1.5);
\draw[red,snake it] (0,-1.0472) to[out angle=180, in angle=0,
    curve through = {(-3,-1.4)}] (-6,-1.5);
    \draw[blue,->,thick] (6,2.0944+0.2)--(0,2.0944+0.2) arc (90:270:0.2)
    --(6,2.0944-0.2)
    node[right] {$\calC_1$};
    \draw[blue,->,thick] (6,0.2)--(0.96242,0.2) arc (90:270:0.2)--(6,-0.2)
    node[right] {$\calC_2$};
\draw[green,->,thick] (6,0)--(3,0) arc (0:180:3)--(-6,0) node[left] {$\calC$};
\end{tikzpicture}
\end{center}
\caption{Henkel contours on the curve $X/\BZ[1/30]$ over $u=\log(x)$-plane.
  The red curves cut of the
  principal branch of $\sqrt{\Delta(\exp(u))}$. When integrating
  this function one
  must change the sign of the integrand when crossing the red lines in order to
  perform the correct analytic continuation.}
\label{fig.calC}
\end{figure}
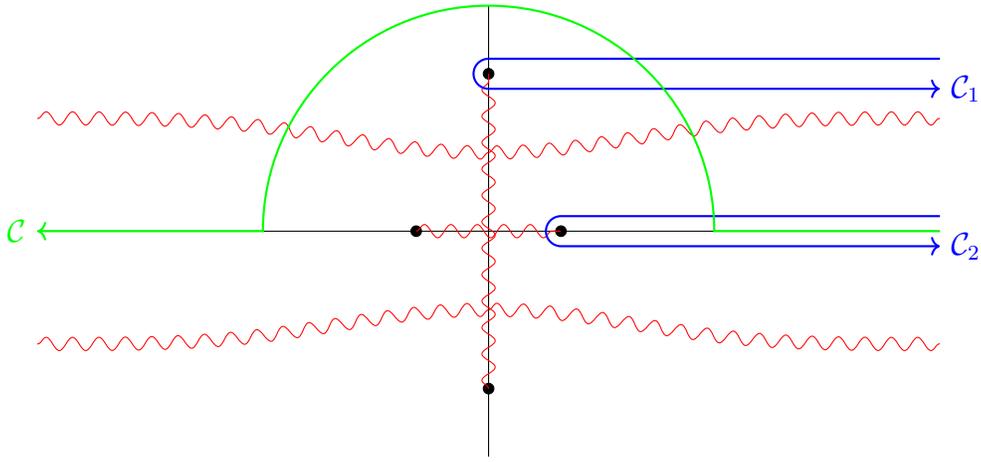

\begin{remark}
\label{rem.res}
The residues of $\omega$ and $\omega'$ at $b_1,\dots, b_4$ is zero
and they vanish at $\infty$.
\end{remark}

We now give more coefficients of ~\eqref{41asy} expanded at $q=1$. 
Using 150 computed values of $f_{X/\mathbb{G}_m}(q)$
and a standard reduction algorithm for computing algebraic de Rham cohomology,
explained for instance in ~\cite{Kedlaya:practice}, we find that
\be
\label{eq:gred}
\begin{aligned}
&g_{X/\BZ[1/30]}(q)
\=\\
\Big(&1
+\frac{227}{2\cdot 3^3\cdot 5^3\cdot 23}(q-1)^4
-\frac{227}{3^3\cdot 5^3\cdot 23}(q-1)^5
+\frac{11851373}{2^2\cdot 3^7\cdot 5^5\cdot 7\cdot 23}(q-1)^6\\
&-\frac{1125623}{2^2\cdot 3^6\cdot 5^5\cdot 7\cdot 23}(q-1)^7
+\frac{2005432343}{2\cdot 3^9\cdot 5^8\cdot 7\cdot 23}(q-1)^8
-\frac{16050504497}{2\cdot 3^9\cdot 5^8\cdot 7\cdot 23}(q-1)^9\\
&+\frac{2736919224300847}{2^3\cdot 3^{12}\cdot 5^{11}\cdot 7\cdot 11\cdot 23}
(q-1)^{10}+\cdots\Big)\omega + \\
\Big(&(q-1)^2-(q-1)^3
+\frac{104713}{2\cdot 3^3\cdot 5^3\cdot 23}(q-1)^4
-\frac{27088}{3^3\cdot 5^3\cdot 23}(q-1)^5\\
&+\frac{207057241}{2^2\cdot 3^7\cdot 5^5\cdot 7\cdot 23}(q-1)^6
-\frac{129006991}{2^2\cdot 3^6\cdot 5^5\cdot 7\cdot 23}(q-1)^7
+\frac{199984079731}{2\cdot 3^9\cdot 5^8\cdot 7\cdot 23}(q-1)^8\\
&+\frac{163800372701}{2\cdot 3^9\cdot 5^8\cdot 7\cdot 23}(q-1)^9
-\frac{34766652291936001}{2^3\cdot 3^{12}\cdot 5^{11}\cdot 7\cdot 11\cdot 23}
(q-1)^{10}+\cdots\Big) \omega'\in H^{1}_\dR(X)[\![q-1]\!]\\
\end{aligned}
\ee


Thus, $g_{X/\BZ[1/30]}(q)$ are linear combinations of $\omega$ and
$\omega'$ as predicted from de Rham cohomology. However, although the forms are
defined over $\BZ$ when considered on the affine chart, their reduction to forms on
the projective variety involves rational numbers. In the example above, 
aside from the spurious denominator $23$ which can be absorbed in rescaling $\omega$
and $\omega'$, these series have denominators that come from the bad primes $2,3,5$
but also have slowly increasing denominators from the other primes.
For example, up to $\mathrm{O}(q-1)^{101}$ in the above series has
denominator
$2^6\cdot 3^{147}\cdot 5^{125}\cdot\prod_{p=7\,\text{prime}}^{101}p$. 
The appearance of single power of a good prime in the denominator is
an accident. Indeed, the coefficient of $(q-1)^k$ for $k \geq 120$ is divisible by
$11^2$. Moreover, for the bad prime prime $2$ it appears that valuation of the
coefficient of $(q-1)^k$ is given by $-\big\lfloor\log_2(k-3)\big\rfloor$. The
valuation of $3$ and $5$ behave like $3k/2$ in the limit.

The logarithmic growth of the denominators is described in the following proposition.

\begin{proposition}
\label{prop.den2}  
There exists $M\in\BZ_{>0}$ such that the $p$-valuation
of the $k$-th coefficient appearing in Equation~\eqref{eq:gred} for $p\neq 2,3,5$
is at most $-\log_p M - \log_p k$.
\end{proposition}

\begin{proof}
From the residue formula~\eqref{fsres}, we find that the coefficient of $(q-1)^{k}$
is given by sums of the form
\be
  A(w,z)\sum_{\ell=0}^{\infty}\ell^{m}\Big(\frac{wx}{(1-w)(1-x^2w)}\Big)^{\ell}
\ee
where $A(w,z)$ is a rational function in $w,z$ and $m\leq 2k$.
The sum is written in terms of derivatives of the geometric series and the degree
of its numerator and denominator clearly grow linearly in $k$.
The numerator and denominator of the rational function $A(w,z)$ also have degree
in $x$ growing at most linearly in $k$. This can be seen from the formula
\be
\frac{qwx}{(1-qw)(1-qx^2w)} \=
\frac{wx+(q-1)wx}{(1-w)(1-x^2w)}\sum_{\ell=0}^{\infty}
\Big(\frac{(1-qw)(1-x^2qw)-(1-w)(1-x^2w)}{(1-w)(1-x^2w)}\Big)^{\ell}\,.
\ee
In the reduction of the forms in $H^{1}(X)\otimes\BQ$ for a polynomial of degree
$n$ divided by $\delta^j$, we add an exact form with denominators away from $2,3,5,$
less than or equal to $j+n$.
Therefore, given the degree grows linearly in $k$ say by a constant $M/2$. We see
that the coefficient of $(q-1)^{k}$ can be expressed as an integral polynomial of
degree at most $Mk/2$ divided by an integral polynomial of degree at most $Mk/2$.
Therefore, we see that the denominator of the reduced form away from the primes
$2,3,5$ is bounded by $Mk$.
\end{proof}

The above proposition implies that the series in $q-1$ can be re-expanded around
other roots of unity. Going back to our example~\eqref{eq:gred}, it follows 
that we can re-expand at other roots of unity and check the
associated gluing. We will do this for the $7$th root of unity $\z_7$.

Firstly, the Frobenius endomorphism is given by
\be
\begin{aligned}
\varphi_p(\omega)
&\=
\frac{1}{\delta^p}\sqrt{\frac{\Delta(x)^p}{\Delta(x^p)}}\frac{dx}{x}\,,\\
\varphi_p(\omega')
&\=
\frac{(x^{3p} - x^{2p} - 2x^p + 5 - 2x^{-p} - x^{-2p} + x^{-3p})}{\delta^{7p}}
\sqrt{\frac{\Delta(x)^{7p}}{\Delta(x^p)^7}}\frac{dx}{x}\,,\\
\end{aligned}
\ee
and computing, we find that
\be
\begin{small}
\begin{aligned}
  \varphi_7(\omega)
  &\=
  (2 + 2\!\cdot\!7 + 6\!\cdot\!7^4 + \mathrm{O}(7^5))\omega
  + (3 + 7 + 3\!\cdot\!7^2 + \mathrm{O}(7^5))\omega'\,,\\
  \varphi_7(\omega')
  &\=
  (2\!\cdot\! 7^{-1} + 2\!\cdot\!7 + 4\!\cdot\!7^2 + 5\!\cdot\!7^3 + 6\!\cdot\!7^4 + \mathrm{O}(7^5))
  \omega + (5 + 4\!\cdot\!7 + 6\!\cdot\!7^2 + 6\!\cdot\!7^3 + \mathrm{O}(7^5))\omega'\,.
\end{aligned}
\end{small}
\ee
Therefore, in the $\BQ_7$-basis given by $\om,\om'$ we find that the
Frobenius automorphisms is represented by the matrix
\be
\begin{pmatrix}
  2 + 2\cdot7 + 6\cdot7^4 + \mathrm{O}(7^5) & 3 + 7 + 3\cdot7^2 + \mathrm{O}(7^5)\\
  2\cdot 7^{-1} + 2\cdot7 + 4\cdot7^2 + 5\cdot7^3 + 6\cdot7^4 + \mathrm{O}(7^5) &
  5 + 4\cdot7 + 6\cdot7^2 + 6\cdot7^3 + \mathrm{O}(7^5)
\end{pmatrix}\,.
\ee
This matrix has trace $0$ and determinant $7^{-1}$.

Secondly, computing the asymptotics of
the 3d-index~\eqref{3dsum}  at $\z_7$, we find leading asymptotics with the period
\be
  (9+8\cos(2\pi/7)+4\cos(4\pi/7))\kappa\,.
\ee
This along with the asymptotics of the series~\eqref{3dsum}, which sees asymptotics
with the contour $\calC_1-\calC_2$, we see that this lifts to an identity of forms
\be
  g_{X/\BZ[1/30]}(\z_7)
  \=
  (9+8\cos(2\pi/7)+4\cos(4\pi/7))\omega
  \=
  (-2\z_7^5 - 4\z_7^4 - 4\z_7^3 - 2\z_7^2 + 5)\omega\,.
  \ee

Finally, we can compare the $p$-adic reexpansion of the series
$g_{X/\BZ[1/30]}(q-1)$ at $q=\z_7$ with the value $\varphi_7(g_{X/\BZ[1/30]}(\z_7))$.
We find that
\be
\begin{aligned}
\varphi_7(g_{X/\BZ[1/30]}(\z_7))
\=&
\Big((3 + 2\cdot7 + 6\cdot7^2 + 6\cdot7^3 + 7^4 + \mathrm{O}(7^5))\z_7^5\\
&+ (6 + 4\cdot7 + 5\cdot7^2 + 6\cdot7^3 + 3\cdot7^4 + \mathrm{O}(7^5))\z_7^4\\
&+ (6 + 4\cdot7 + 5\cdot7^2 + 6\cdot7^3 + 3\cdot7^4 + \mathrm{O}(7^5))\z_7^3\\
&+ (3 + 2\cdot7 + 6\cdot7^2 + 6\cdot7^3 + 7^4 + \mathrm{O}(7^5))\z_7^2\\
&+ (3 + 4\cdot7 + 7^2 + 2\cdot7^4 + \mathrm{O}(7^5))\Big)\omega\\
&+ \Big((1 + 4\cdot7 + 6\cdot7^3 + 6\cdot7^4 + \mathrm{O}(7^5))\z_7^5\\
&+ (2 + 7 + 7^2 + 5\cdot7^3 + 6\cdot7^4 + \mathrm{O}(7^5))\z_7^4\\
&+ (2 + 7 + 7^2 + 5\cdot7^3 + 6\cdot7^4 + \mathrm{O}(7^5))\z_7^3\\
&+ (1 + 4\cdot7 + 6\cdot7^3 + 6\cdot7^4 + \mathrm{O}(7^5))\z_7^2\\
&+ (1 + 2\cdot7^2 + 2\cdot7^3 + \mathrm{O}(7^5)) \Big) \omega'\\
\=&g_{X,1}(\z_7-1)\,.
\end{aligned}
\ee
This agrees with the gluing, which should be deducible from the fact that
$g_{X/\BZ[1/30]}(q)\in\calH(X/\BZ[1/30])$.

The above computations illustrate the following theorem.

\begin{theorem}
\label{thm.3}
We have:
\be
\label{Yhab}
g_{X/\BZ[1/30]}(q) \in \calH_\nv^1(X/\BZ[1/30]) \,.
\ee
\end{theorem}

\begin{proof}
Theorem~\ref{thm.2} implies that $\Sf_{X/\mathbb{G}_m}(q) \in \calH_{R/\BZ[x^{\pm1}]}$.
Apply the push-forward Theorem~\ref{thm.push} to the above element times
$\tfrac{dx}{x}$, in other words to $g_{X/\BZ[1/30]}(q)$, to conclude
Equation~\eqref{Yhab}.
\end{proof}

\subsection{Habiro cohomology as a module over the Habiro ring}
\label{sub.41rels}

In this section we discuss how the explicit elements $g_{X/\BZ[1/30]}(q)$
in the Habiro cohomology $\calH_\nv^1(X/\BZ[1/30])$ given in Theorem~\ref{thm.3}
generate a module of finite rank over the Habiro ring $\calH_{\BZ[1/30]}$. 

Indeed, consider the elements
\be
  g_{m}(q)
  \=
  x^m \Sf_{X/\mathbb{G}_m}(q) \frac{dx}{x}\,.
\ee
These live in a strictly larger space of forms which have poles and residues at
$x=0,\infty$ for various $m\in\BZ$. Taking suitable linear combinations, we give
elements in $\calH_\nv^1(X/\BZ[1/30])$. In particular, we obtain
\be
\label{arel}
\begin{aligned}
  a_0(q)&\=g_0(q)\\
  &\=[\om]+(q-1)^2[\om']+\cdots\,,\\
  a_1(q)&\=g_1(q)-g_2(q)\\
  &\=\Big(-\frac{9}{2^2\cdot 23}+\frac{9}{2^2\cdot 23}(q-1)^2\Big)[\om]+
  \Big(\frac{10125}{2^2\cdot 23}+\frac{639}{2^2\cdot 23}(q-1)^2\Big)[\om']+\cdots\,,\\
  a_2(q)&\=g_4(q)-2g_3(q)+(1+q^{-1}(q-1)^2)g_1(q)\\
  &\=\Big(-\frac{89}{2^3\cdot 23}+\frac{35}{2^3\cdot 23}(q-1)^2\Big)[\om]+
  \Big(-\frac{3375}{2^3\cdot 23}-\frac{4139}{2^3\cdot 23}(q-1)^2\Big)[\om']+\cdots\,.
\end{aligned}
\ee
Then, we find that three elements $a_j(q)$ for $j=0,1,2$ satisfy a relation
of the form
\be
a_2(q) \=
\Big(-\frac{1}{2}+\frac{7}{40}(q-1)^2+\cdots\Big)a_0(q) +
\Big(-\frac{1}{6}-\frac{71}{360}(q-1)^2+\cdots\Big)a_1(q)\,.
\ee
It is remarkable that although 
the coefficients of the series $a_0(q)$ and $a_1(q)$ live in $\BQ[\![q-1]\!]$,
the coefficients of the above relation live in $\BZ[1/30][\![q-1]\!]$.
In fact we find that, to order $\mathrm{O}(q-1)^{150}$, their denominator is
$2^{151}3^{224}5^{187}$. This gives numerical supporting evidence
that these coefficients are elements of the Habiro ring $\calH_{\BZ[1/30]}$.

\subsection{A family of elliptic curves associated to the $4_1$ knot}
\label{sub.41lambda}

In this section, we are no longer motivated from the 3D-index, but instead
of the question of constructing elements in the Habiro cohomology of a family
of elliptic curves associated to the $4_1$ knot.

To achieve this, we use a variable $\lambda$ (as is traditionally called in
the classical Legendre family of elliptic curves) and consider the $\lambda$-deformed
series $h_n(\lambda,q)$ given by

\be
\label{hnladef}
h_{n}(\lambda,q) \= \sum_{k\in\BZ}
(-1)^{k+n}\frac{q^{k(k+1)/2+nk+n(2n+1)/2}}{(q;q)_{k}(q;q)_{k+2n}} \lambda^{n+k},
\qquad (n \in \BZ) \,.
\ee
With the choice of variables $(x,y)=(q^n,q^k)$, 
the corresponding version of the Nahm equation~\eqref{nahmdesc} now becomes
\be
\label{nahmla}
\lambda xy+(1-x^2y)(1-y)\=0 
\ee
and defines a family $E_\lambda$ of affine elliptic curves parametrised by $\lambda$ 
as well as a smooth map $E/B_x$ with $E=\Spec(\mathsf{R})$ and
$B_x=\Spec(\BZ[x^{\pm 1},\lambda^{\pm 1}])$
determined by the \'etale map
\be
\label{41etla}
\BZ[x^{\pm 1},\lambda^{\pm 1}]\to \mathsf{R} \=
\BZ[x^{\pm 1},\lambda^{\pm 1},\tfrac{1}{\delta}][y^{\pm 1}]/(\lambda xy+(1-x^2y)(1-y))
\ee
with
\be
\label{delta2la}
\delta^2(\lambda,x)\=x^{2} - 2 \lambda x +(\lambda^2-2) - 2 \lambda x^{-1} + x^{-2}\,.
\ee
is a $\lambda$-deformation of~\eqref{delta2}.

As before, from~\eqref{hnladef} one obtains $f_{E/B_x}(q)$
and its symmetrisation 
\be
\label{fsladef}
\Sf_{E/B_x}(q) = f_{E/B_x}(q)
f_{E/B_x}(q^{-1})
\ee
which may also be given by the residue formula
\be
\label{fslares}
\Sf_{E/B_x}(q) \= \sum_{z^m=y}
\underset{w=z}{\Res}
\sum_{k=0}^{\infty}
(-1)^{k}\frac{q^{k(k+1)/2}w^kx^k}{(qw;q)_{k}(qwx^2;q)_{k}} \lambda^{k-1}
\frac{dw}{w}
\ee
where $y$ satisfies~\eqref{nahmla}.

The proof of Theorem~\ref{thm.2} implies the following.

\begin{theorem}
\label{thm.2la}  
$\Sf_{E/B_x}(q)$ is an element of the Habiro ring
$\calH_{\mathsf{R}/\BZ[x^{\pm 1},\lambda]}$.
\end{theorem}

For example, we have the following $\lambda$-deformation of~\eqref{fsymval}
\be
\label{fsymlaval}
\begin{aligned}
  \Sf_{E/B_x}(q) \= & \frac{1}{\delta}\big(1+
  \frac{(q-1)^2}{\delta^{6}}(\lambda x^3 - \lambda^2 x^2 - \lambda(1+\lambda^2) x
  + \lambda^2(4+\lambda^2) \\ &  - \lambda(1+\lambda^2) x^{-1} - \lambda^2 x^{-2}
  + \lambda x^{-3})
+ \dots \big)
\end{aligned}
\ee
We have computed $15$ values of the above series. 


Next, we discuss the summed symmetrisation of $h_n(\lambda,q)$, which is the
series given by
\be
\label{3dlasum}
  \sum_{n\in\BZ}
  h_{n}(\lambda,q)h_{n}(\lambda,q^{-1})
  \=
  \sum_{n,k,\ell\in\BZ}
  (-1)^{k+\ell}\frac{q^{k(k+1)/2+\ell(\ell+1)/2+nk+n\ell+n(2n+1)}}{
    (q;q)_{k}(q;q)_{k+2n}(q;q)_{\ell}(q;q)_{\ell+2n}} \lambda^{2n+k+\ell}\,.
\ee
The Nahm equations now become
\be
\label{xyznahmla}
  \lambda^2\frac{x^4yz}{(1-x^2y)^2(1-x^2z)^2}\=1\,, \quad 
  -\lambda \frac{xy}{(1-x^2y)(1-y)}\=1\,,\quad
  -\lambda \frac{xz}{(1-x^2z)(1-z)}\=1\,.
\ee
  
As before, summation over $n$ geometrically corresponds to the push-forward 
along $x$ to a family 
$E/B$ of elliptic curves parametrised by $\lambda$
over $B=\mathrm{Spec}(\BZ[\lambda, \tfrac{1}{\lambda(\lambda^2-16)}])$.

Letting
\be
\label{gxladef}  
g_{E/B}(q) \= \Sf_{E/B}(q) \frac{dx}{x} \,,
\ee
and using the $p$-Frobenius $\varphi_p(\lambda)=\lambda^p$, 
we obtain the following theorem whose proof follows verbatim from the proof
of Theorem~\ref{thm.3}.

\begin{theorem}
\label{thm.3la}
We have:
\be
\label{Yhabla}
g_{E/B}(q) \in \calH_\nv^1(E/B) \,.
\ee
\end{theorem}

The first few values of the expansion of $g$ at $q=1$ is given by
\be
\label{gla}
\begin{tiny}
\begin{aligned}
&g_{E/B}(q)
\=
\Big(1 + \frac{-4\lambda^6 - 283\lambda^4 - 848\lambda^2}{
10 (\lambda^2-16)^3 (7 \lambda^2 + 16)
}(q-1)^4 + \cdots\Big)\om
\\
&+\Big((q-1)^2 - (q-1)^3 + \frac{74\lambda^8 - 551\lambda^6 + 95664\lambda^4
  - 16640\lambda^2 - 602112}{
  10 (\lambda^2-16)^3 (7 \lambda^2 + 16)}
(q-1)^4 + \cdots\Big)\om'\in H^{1}(E/B)[\![q-1]\!]\,.
\end{aligned}
\end{tiny}
\ee


\subsection{Effective computation of linear $q$-difference equations}
\label{sub.eff}

As mentioned in the introduction, the explicit classes in the Habiro cohomology
as expressed one way or another by proper $q$-hypergeometric multisums, see
for example~\eqref{hnladef}, \eqref{hnlasym}, \eqref{3dlasum}. The main theorem
of Wilf--Zeilberger is that proper $q$-hypergeometric sums are $q$-holonomic, in all
of their variables~\cite{WZ}. What's more, their theorem comes with an explicit
description of these $q$-holonomic modules which has been implemented by many
authors (see e.g.,~\cite{PWZ}), and among them by Koutschan~\cite{Koutschan} whose
method we will use. 

For example, the function $h_n(\lambda,q)$ of Equation~\eqref{hnladef} is a
proper $q$-hypergeometric multisum, hence $q$-holonomic in the variables $(n,\lambda)$.
The same conclusion holds for the symmetrisation $\Sh_n(\lambda,q)$ of 
$h_n(\lambda,q)$
\be
\label{hnlasym}
\Sh_n(\lambda,q) \= 
h_n(\lambda,q) h_n(\lambda,q^{-1})
\ee
which is also $q$-holonomic in the variables $(n,\lambda)$.

The \texttt{qHolonomic} package of Koutschan~\cite{Koutschan} computes the
annihilator of $h_n(\lambda,q)$ and of its symmetrisation, and 
implies that the function $\Sh_n(\lambda,q)$ satisfies the linear
$q$-difference equation

\begin{small}
\be
\begin{aligned}
\label{qdiffhx}
  & q^{2n} (1 - \lambda q^{2+n} + q^{2n}) h_n(\lambda,q) \\
  & + (-1 + \lambda q^{1+n} - q^{2n}) 
(1 - \lambda q^{1+n} - \lambda q^{2+n} + q^{2n} + \lambda^2 q^{3+2n}
- \lambda q^{1+3n} - \lambda q^{2+3n} + q^{4n}) h_n(\lambda q,q)  \\ &
- (-1 + \lambda q^{2+n} - q^{2+n}) (1 - \lambda q^{1+n} - \lambda q^{2+n} + q^{2n} + 
\lambda^2 q^{3+2n} - \lambda q^{1+3n} - \lambda q^{2+3n} + q^{4n})
h_n(\lambda q^2,q) \\ &
- q^{2n} (1 - \lambda q^{1+n} + q^{2n}) h_n(\lambda q^3) \= 0 \,. 
\end{aligned}
\ee
\end{small}


A fortiori, it follows that the element $\Sf_{E/B}(q)$
of $\calH_\nv^1(E/B)$ satisfies the linear $q$-difference equation
~\eqref{qdiffhx} with $x=q^n$.

Finally, we discuss the series of Equation~\eqref{3dlasum}, which is also a 
proper $q$-hypergeometric multisum and hence $q$-holonomic. In this case,
the \texttt{qHolonomic} package computes a 6th order linear $q$-difference equation
\be
\label{order6}
\sum_{j=0}^{10} a_j(\lambda,q) h(q^j \lambda,q) \=0
\ee
for the series~\eqref{3dlasum}, and a fortiori for the elements
$g_{E/B}(q)$ of the Habiro cohomology $\calH_\nv^1(E/B)$, with respect to
$\lambda$, where

\begin{small}
\be
\label{a06}
\begin{aligned}
a_0 \= & q^6 (-1 - q + \lambda q^4) (1 + q + \lambda q^4) 
\\
a_1 \= & -q^4 (-2 - 6 q - 8 q^2 - 6 q^3 - 2 q^4 - \lambda^2 q^4 - 3 \lambda^2 q^5 - 
3 \lambda^2 q^6 - 2 \lambda^2 q^7 + 2 \lambda^2 q^8 + 3 \lambda^2 q^9
\\ & + 2 \lambda^2 q^{10} + 
\lambda^4 q^{12} + \lambda^4 q^{13}) \\ 
a_2 \= &
q^2 (-1 - 6 q - 14 q^2 - 18 q^3  - 14 q^4 - 6 q^5 - q^6 + 3 \lambda^2 q^6 + 
   8 \lambda^2 q^7 + 10 \lambda^2 q^8 + 10 \lambda^2 q^9 \\ & + 9 \lambda^2 q^{10} + 
   6 \lambda^2 q^{11} - \lambda^4 q^{11} + \lambda^2 q^{12} - 2 \lambda^4 q^{12} 
   - 2 \lambda^4 q^{13} - 
   4 \lambda^4 q^{14} - 3 \lambda^4 q^{15} + \lambda^6 q^{19}) 
\\
a_3 \= &   
-q (-2 - 8 q - 18 q^2 - 24 q^3 - 18 q^4 - 8 q^5 - 2 q^6 + 
3 \lambda^2 q^6 + 11 \lambda^2 q^7 + 23 \lambda^2 q^8 + 26 \lambda^2 q^9 \\ &
+ 23 \lambda^2 q^{10} + 11 \lambda^2 q^{11} + 3 \lambda^2 q^{12} - 3 \lambda^4 q^{13}
- 7 \lambda^4 q^{14} - 8 \lambda^4 q^{15} - 7 \lambda^4 q^{16} - 3 \lambda^4 q^{17}
+ 2 \lambda^6 q^{21}) 
\\   
a_4 \= &   (-1 - 6 q - 14 q^2 - 18 q^3 - 14 q^4 - 6 q^5 - q^6 + \lambda^2 q^6
   + 6 \lambda^2 q^7 + 9 \lambda^2 q^8 + 10 \lambda^2 q^9 + 10 \lambda^2 q^{10}
   \\ &
   + 8 \lambda^2 q^{11} + 3 \lambda^2 q^{12} - 
   3 \lambda^4 q^{15} - 4 \lambda^4 q^{16} - 2 \lambda^4 q^{17}
   - 2 \lambda^4 q^{18} - \lambda^4 q^{19} + 
   \lambda^6 q^{23}) 
\\
a_5 \= &   (2 + 6 q + 8 q^2 + 6 q^3 + 2 q^4
   - 2 \lambda^2 q^6 
   - 3 \lambda^2 q^7 - 2 \lambda^2 q^8 + 2 \lambda^2 q^9 + 3 \lambda^2 q^{10}
   + 3 \lambda^2 q^{11} \\ &
   + \lambda^2 q^{12} - \lambda^4 q^{15} - \lambda^4 q^{16}) 
\\
a_6 \= &   
 (-1 - q + \lambda q^3) (1 + q + \lambda q^3) f(\lambda q^6) \,.
\end{aligned}
\ee
\end{small}

The equation~\eqref{order6} has $(q,\lambda,\text{shift})$-degree $(23,6,6)$.
Its Newton polygon is shown in Figure~\ref{fig:newton}. 

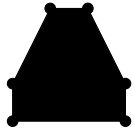
\begin{figure}[htpb!]
\begin{center}  
\begin{tikzpicture}[scale=0.25]
\filldraw (0,0) circle (8pt);
\filldraw (0,2) circle (8pt);
\filldraw (2,6) circle (8pt);
\filldraw (4,6) circle (8pt);
\filldraw (6,2) circle (8pt);
\filldraw (6,0) circle (8pt);
\filldraw[opacity=0.2] (0,0) -- (0,2) -- (2,6) -- (4,6) -- (6,2) -- (6,0) -- cycle;
\end{tikzpicture}
\caption{The Newton polygon of the 6th order operator ~\eqref{order6}.}
\label{fig:newton}
\end{center}
\end{figure}

The edge polynomial (in the $w$-variable) of the bottom horizontal edge $e_0$ 
of length 6, the edge $e_1$
of length 2 and slope 2, the top horizontal edge $e_2$ of length 2 and
the edge $e_3$ of length 2 and slope -2 are given by
\be
\label{edgepoly}
\begin{aligned}
  e_0 : & \,\, (q - w)^2 (q^2 - w)^2 (-1 + w)^2 \\
  e_1 : & \,\, 1 - q^2 w - q^3 w + q^7 w^2 \\
  e_2 : & \,\, (-1 + q w)^2 \\
  e_3 : & \,\,  q^{17} - q^9 w - q^{10} w + w^2 \,.
\end{aligned}
\ee

We did two sanity checks of the equation~\eqref{order6}. Namely, 
we computed the series $g_{E/B}(q) + O(q-1)^{12}$ (whose first
two terms are given in~\eqref{gla}), as well as the series
$g_{E/B}(q) + O(q-1)^{50}$ when $\lambda=q^\ell$ for
$\ell=0,\dots,10$ and verified that they both satisfy Equation~\eqref{order6}.

\subsection{$q$-deformation of the Picard-Fuchs equation}
\label{sub.q=1}

The $q=1$ specialization of the linear $q$-difference equation~\eqref{order6}
is the second order linear $q$-difference equation
\be
\label{PF41}
(\lambda-4) \lambda (\lambda+4)f''(\lambda) + (-16 + 3 \lambda^2) f'(\lambda)
+ \lambda f(\lambda) = 0\,.     
\ee
This is the Picard-Fuchs equation for the Gauss--Manin connection of this
elliptic family.

So, the fact that our explicit elements in $H^1_\dR(E/B)$
generate a $q$-holonomic module with respect to the variable $\lambda$,
implies the following.

\begin{corollary}
\label{cor.41}
Our explicit elements in the first Habiro cohomology of a family of elliptic curves
gives a canonical $q$-deformation of its Picard-Fuchs equation.
\end{corollary}

By canonical, we mean a linear $q$-difference equation with coefficients in
$\BZ[\lambda,q]$ which is of smallest order and with coefficients of content $1$.

At this point, one might expect that the 6th order operator~\eqref{order6} has
a second order right linear factor that the element of
$g_{E_{4_1,\lambda}/B_{\lambda}}(q)$ of the Habiro cohomology satisfies, whose
specialization to $q=1$ is the Picard-Fuchs equation of the family
$E/B$. After all, the package \texttt{qHolonomic} does not
guarantee to find a minimal order $q$-difference operators, and in addition the
Habiro cohomology $\calH_\nv^1(E/B)$ is a module of
rank $2$ over the Habiro ring
$\calH_{\BZ[\lambda,\tfrac{1}{\lambda(\lambda^2-16)}]/\BZ[\lambda]}$. 

However, after explicitly computing the series
$g_{E/B}(q) +O(q-1)^{100}$ for $\lambda=q^n$ for $n=-1,0,1$
and solving as was done in Equation~\eqref{arel}, we found two power series in $q-1$
that ought to be in $\BQ(q)$, with 50 coefficients known.
In that case, Pade
approximation failed to recognize them as rational functions, and what's more, the
coefficients of these series experimentally (and up to the 100th coefficient in
powers of $q-1$) appear to be growing factorially, and not exponentially.
This convinced us that there is no second order linear $q$-difference equation for
$g_{E/B}(q)$.

Incidentally, there is a computable sufficient criterion to prove that the 6th order
operator has no second order right factor, explained in ~\cite{GK:74}. The 
criterion is to compute the second exterior power and show that it does not have
a linear factor. Both tasks are possible and their implementation is explained in
detail in~\cite{GK:74}, but the size of the output is formidable. The second
exterior power of the 6th order operator~\eqref{order6} is a
$q$-difference operator of order $15$, and using the methods of~\cite{GK:74},
we found that the exponent of its coefficients with respect to
$(q,\lambda,\text{shift})$ is $(615,64,15)$.


Finally, we mention that canonical $q$-deformation of the Picard-Fuchs equation
of the quintic family was given explicitly in~\cite[Lem.1.6]{GS:quintic}, see
also Wen~\cite{Wen}, where the $q$-deformation was obtained from the genus 0 quantum
$K$-theory of the quintic. As it turns out, the $q$-deformation of the 4th order
Picard-Fuchs equation was a 24th order $q$-difference operator (see
e.g.,~\cite[Eqn.(1.3)]{Wen}). 

\bibliographystyle{plain}
\bibliography{biblio}

\end{document}